\documentclass[a4paper,12pt]{article} 
\title{N\'eron models of jacobians over base schemes of dimension greater than 1}

\usepackage{amsmath, amssymb, mathrsfs, amsthm, geometry, verbatim, mathtools}
\usepackage{hyperref}
\usepackage[all]{xy}
\usepackage[OT2, T1]{fontenc}

\usepackage{pgf,tikz}
\usetikzlibrary{arrows}

\usepackage{cleveref}

\newcommand{\on}[1]{\operatorname{#1}}
\newcommand{\bb}[1]{{\mathbb{#1}}}
\newcommand{\cl}[1]{{\mathscr{#1}}}
\newcommand{\ca}[1]{{\mathcal{#1}}}

\newcommand{\abs}[1]{\lvert#1\rvert}

\newcommand{\ra}{\rightarrow}

\newcommand{\hra}{\hookrightarrow}
\newcommand{\sub}{\subseteq}
\newcommand{\tra}{\rightarrowtail}

\theoremstyle{definition}
\newtheorem{definition}{Definition}[section]

\theoremstyle{plain}
\newtheorem{proposition}[definition]{Proposition}
\newtheorem{lemma}[definition]{Lemma}
\newtheorem{theorem}[definition]{Theorem}
\newtheorem{corollary}[definition]{Corollary}

\theoremstyle{remark}

\newtheorem{remark}[definition]{Remark}

\newtheorem{example}[definition]{Example}

\def\defeq{\stackrel{\text{\tiny def}}{=}}

\renewcommand{\phi}{\varphi}
\newcommand{\closure}[2]{\on{clo}_{#1}(#2)}

\newcommand{\Pictdz}{\on{Pic}^{[0]}}
\newcommand{\Picz}{\on{Pic}^{0}}

\newcounter{temp}

\author{David Holmes}
\date{\today}

\newcounter{nootje}
\newcommand{\shorten}[2]{#2} 

\newcommand{\beq}{\begin{equation}}
\newcommand{\eeq}{\end{equation}}
\newcommand{\beqs}{\begin{equation*}}
\newcommand{\eeqs}{\end{equation*}}

\begin{document}
\maketitle
\begin{abstract} 
We investigate to what extent the theory of N\'eron models of jacobians and of abel-jacobi maps extends to relative curves over base schemes of dimension greater than 1. We give a necessary and sufficient criterion for the existence of a N\'eron model. We use this to show that, in general, N\'eron models do not exist even after making a modification or even alteration of the base. On the other hand, we show that N\'eron models \emph{do} exist outside some codimension-2 locus. 

\end{abstract}

\tableofcontents

\newcommand{\et}{^{\on{et}}}

\section{Introduction}

Let $C/S$ be a proper generically smooth curve over a Dedekind scheme. The jacobian $J$ of the generic fibre of $C$ has a N\'eron model $N/S$ (\cite{Neron1964Modeles-minimau}, \cite{Bosch1990Neron-models}). If a point $\sigma \in C(K)$ is given, we obtain a canonical map $\alpha\colon C^{\on{sm}} \ra N$ from the smooth locus of $C/S$ to the N\'eron model; this map is called the \emph{abel-jacobi} map. If moreover $C/S$ is for example semistable, then it is possible to give a very explicit description of the N\'eron model $N$ and the abel-jacobi map $\alpha$, using the relative Picard functor of $C/S$ (see for example \cite{Bosch1990Neron-models}, \cite{Edixhoven1998On-Neron-models}). 

Suppose now that we replace the Dedekind scheme $S$ by an arbitrary regular scheme. The main aim of this paper is to understand to what extent the theory of N\'eron models and abel-jacobi maps described above carries across into this more general situation, and in what ways the theory must be modified. First we give a definition of the N\'eron model\footnote{Perhaps a more conventional (though less general) definition would be to assume $S$ to be integral and to have $A$ an abelian scheme over the generic point of $S$. However, such $A$ would extend to an abelian scheme over some non-empty open $U \sub S$, and an abelian scheme over a regular integral base is automatically the N\'eron model of its generic fibre. As such, the two definitions are roughly equivalent. }:
\begin{definition}
Let $S$ be a scheme and $U\sub S$ a schematically dense open subscheme. Let $A/U$ be an abelian scheme. A \emph{N\'eron model} for $A$ over $S$ is a smooth separated algebraic space $N/S$ together with an isomorphism $A \ra N_U = N \times_S U$  satisfying the following universal property:
\emph{let $T \ra S$ be a smooth morphism of algebraic spaces\footnote{Our non-existence results still hold, with almost the same proofs, if `algebraic spaces' here is replaced by `schemes'.} and $f\colon T_U \ra N_U$ any $U$-morphism. Then there exists a (unique) $S$-morphism $F\colon T \ra N$ such that $F|_{T_U} = f$. }
\end{definition}
Note that if a N\'eron model exists then it is unique up to unique isomorphism, and has a canonical structure as a group algebraic space. What we here call a `N\'eron model' is more commonly called a `Neron \emph{lft} model' (cf. \cite{Bosch1990Neron-models}), where \emph{lft} stands for `locally of finite type', to emphasise that quasi-compactness of $N/S$ is not assumed in the definition. 

Let $C/S$ be a semistable curve which is smooth over a dense open subscheme $U \sub S$. Write $J$ for the jacobian of the smooth proper curve $C_U/U$. If $C/S$ is pointed, a N\'eron model of the jacobian admits an abel-jacobi map from the smooth locus of $C/S$ (by the universal property). Our first result is essentially negative in nature. We say that such a curve $C/S$ is \emph{aligned} if it satisfies a certain rather strong condition described in terms of the dual graphs of fibres and the local structure of the singularities (\cref{def:balance}). We have
\begin{theorem}[\cref{thm:balanced_equivalent_to_flat}, \cref{lem:NM_exists}]\label{thm:intro_NM}
Suppose $S$ is regular. We have:
\begin{enumerate}
\item if the jacobian $J/U$ admits a N\'eron model over $S$ then $C/S$ is aligned; 
\item if $C$ is regular and $C/S$ is aligned then the jacobian $J/U$ admits a N\'eron model over $S$.  
\end{enumerate}

\end{theorem}

One consequence of the above theorem is that if $C/S$ is not aligned then no proper regular semistable model of $C_U$ over $S$ is aligned. If $C/S$ is split semistable and smooth over the complement of a strict normal crossings divisor in $S$ then we can determine whether or not a N\'eron model exists without needing to construct a regular model - see \cref{sec:changing_the_model}. As a corollary of the above theorem, we have
\begin{corollary}[\cref{cor:NM_outside_codimension_2}]Let $S$ be an excellent scheme, regular in codimension 1. Let $U \sub S$ be a dense open regular subscheme, and let $C/S$ be a semistable curve which is smooth over $U$. Then there exists an open subscheme $U \sub V \sub S$ such that the complement of $V$ in $S$ has codimension at least 2 in $S$ and such that the jacobian of $C_U/U$ admits a N\'eron model over $S$ whose fibrewise-identity-component is an $S$-scheme. 
\end{corollary}


It is easy to construct examples (c.f. \cref{sec:basic_examples}) of non-aligned semistable curves; intuitively, `most' semistable curves over base schemes of dimension greater than 1 are non-aligned, and so do not admit N\'eron models (one sees easily from the definition that a semistable curve over a Dedekind scheme must be aligned). However, perhaps inspired by \cite{Deligne1985Le-lemme-de-Gab}, one might ask whether $C/S$ admits a N\'eron model after blowing up $S$, or making some alteration of $S$. This turns out again to have a negative answer; we have
\begin{theorem}[\cref{prop:persistence}]\label{thm:intro_balance_stable}
Let $S$ be a locally noetherian normal scheme, $C/S$ semistable over $S$ and smooth over a dense open $U \sub S$, and $f\colon S' \ra S$ a proper surjective morphism. Then $C/S$ is aligned if and only if $f^*C/S'$ is aligned.  
\end{theorem}
Combining with \cref{thm:intro_NM} we see that a non-aligned semistable jacobian does not admit a N\'eron model even after blowing up or altering the base scheme. 

For many applications, rather than needing the full strength of the N\'eron mapping property, one only needs to extend (a multiple of) a single section of the jacobian $J$ to some semiabelian scheme over $S$ (maybe after an alteration of $S$).  However, it turns out that this does not make the problem easier; in general extending a section in this way is as hard as constructing a N\'eron model. We illustrate this in \cref{sec:counterexample} with an example.



\subsection{Idea of the proof: N\'eron models via the relative Picard functor}
Our existence and non-existence results for the N\'eron model proceed via an auxiliary object, the `total-degree-zero relative Picard functor' $\on{Pic}_{C/S}^{[0]}$ (see \cref{sec:pic}), a group-algebraic-space. The base-change of $\on{Pic}_{C/S}^{[0]}$ to the locus $U \sub S$ where $C$ is smooth coincides with the Jacobian of $C_U/U$. It is also rather easy to show that  $\on{Pic}_{C/S}^{[0]}$ satisfies the `existence' part of the N\'eron mapping property whenever $C$ is regular. However, if $C/S$ has non-irreducible fibres then $\on{Pic}_{C/S}^{[0]}$ is in general highly non-separated, and for this reason it fails to satisfy the `uniqueness' part of the N\'eron mapping property. As such, we wish to construct some kind of `separated quotient' of $\on{Pic}_{C/S}^{[0]}$, in the hope that this will satisfy the whole N\'eron mapping property. 

Now the failure of separatedness in a group-space is measured by the failure of the unit section to be a closed immersion; as such, a natural way to construct a `separated quotient' of $\on{Pic}_{C/S}^{[0]}$ is to quotient by the closure of the unit section. Quotients by \emph{flat} subgroup-spaces always exist in the category of algebraic spaces, but the problem is that the closure of the unit section in $\on{Pic}_{C/S}^{[0]}$ (which we shall refer to as $\bar{e}$) is not in general flat over $S$, and so this quotient does not exist\footnote{One might be tempted to apply \cite{Raynaud1971Criteres-de-pla} to flatten $\bar{e}$ by blowing up $S$, but note that $\on{Pic}_{C/S}^{[0]}$ is not quasi-compact over $S$, and so their results do not apply; indeed, one consequence of our results is that, if $C/S$ is not aligned, then $\bar{e}$ does not become flat after \emph{any} modification of $S$. }. By a slightly more delicate argument, one can even show that the existence of a N\'eron model is \emph{equivalent} to the flatness of $\bar{e}$. The technical heart of this paper is the proof that the flatness of $\bar{e}$ over $S$ is equivalent to the alignment of $C/S$ (our combinatorial condition introduced in \cref{sec:definition_of_balance}); this is carried out in \cref{sec:Cartier_classification} and \cref{sec:balance_equivalent_to_flatness}. 

\subsection{Applications}
Here we briefly mention four applications of the theory developed in this paper. 

\textbf{Height jumping}

\noindent In \cite{Hain2013Normal-function}, Richard Hain defined an invariant, the height jump, associated to a morphism of variations of mixed Hodge structures, and conjectured that in certain cases this invariant should be non-negative. In \cite{David-Holmes2014Neron-models-an} we use the theory of N\'eron models and labelled graphs developed in this paper, along with a study of resistances in electrical networks, to prove some cases of this conjecture. We use the same techniques to extend the definition of the height jump to an algebraic setting. For 1-dimensional families of abelian varieties we show also in \cite{David-Holmes2014Neron-models-an} that the algebraic height jump is actually \emph{bounded}, and we use this to give a new proof of a theorem of Silverman and Tate (\cite{Silverman1983Heights-and-the}, \cite{Tate1983Variation-of-th}). 

\textbf{Bounds on orders of torsion points} 

\noindent In \cite{Holmes2014A-note-on-Neron} we connect the existence of N\'eron models for families of jacobians to the question of bounding the orders of rational torsion points on abelian varieties, extending the work on algebraic height jumping in \cite{David-Holmes2014Neron-models-an}. 

\textbf{The case of universal curves}

\noindent The universal stable pointed curve over the Deligne-Mumford-Knudsen compactification $\overline{\ca{M}}_{g,n}$ of the moduli stack of pointed curves is never aligned unless $g=0$ or $n=g=1$. As such, a N\'eron model of the universal jacobian does not in general exist. In \cite{Holmes2014A-Neron-model-o}, we construct a morphism $\beta\colon \widetilde{\ca{M}}_{g,n} \ra \overline{\ca{M}}_{g,n}$ over which a N\'eron model of the universal jacobian does exist, and such that $\beta$ is universal among morphisms with this property. 

\textbf{Connection to logarithmic geometry}

\noindent In \cite{Bellardini2015On-the-Log-Pica}, Bellardini adapts the definition of alignment to the setting of logarithmic geometry, and shows that a family of curves is aligned if and only if it is log cohomologically flat and the logarithmic Picard functor is separated. In particular, this gives a way to verify in examples that a family of curves is log cohomologically flat.

\subsection{Acknowledgements}
The author would like to thank Jos\'e Burgos Gil for pointing out a serious mistake in a previous iteration of this work, to thank Hendrik Lenstra for supplying some of the ideas behind the proof of the main result of \cref{sec:Cartier_classification}, and to thank Owen Biesel, Raymond van Bommel, K\k{e}stutis \v{C}esnavi\v{c}ius, Bas Edixhoven, Ariyan Javanpeykar, Robin de Jong, Qing Liu and Morihiko Saito for helpful comments and discussions. The author is also extremely grateful to two anonymous referees. The first referee produced two very thorough and rapid reports highlighting several serious mistakes in earlier versions of this paper. As well as pointing out several additional issues, the second referee pointed out several places where proofs could be improved and streamlined. Most dramatically, the second referee provided a new proof of \cref{CDthm:main}, based on the original one but about half as a long and far more readable, and kindly allowed this new proof to be included in this paper. Both of these contributions have dramatically improved the correctness and readability of this article. Finally, the author would like to thank the editor Daniel Huybrechts for his patience and encouragement with the numerous iterations of this paper. 

\section{Definition and basic properties of aligned curves}\label{sec:definition_of_balance}
\newcommand{\mono}{\textbf{Mon}_0}


\begin{definition}
A (finite) graph is a triple $(V,E,\on{ends})$ where $V$ and $E$ are finite sets (vertices and edges), and $\on{ends}\colon E \ra (V\times V)/\on{S}_2$ is a function, which we think of as assigning to each edge the unordered pair of its endpoints. A \emph{loop} is an edge whose endpoints are the same. A \emph{circuit} is a path of positive length which starts and ends at the same vertex, and which does not repeat any edges or any other vertices. 

Let $L$ be a monoid or set, and $\Gamma$ a graph. An \emph{edge-labelling} of $\Gamma$ by $L$ is a function $\ell$ from the set of edges of $\Gamma$ to $L$. We call the pair $(\Gamma, \ell)$ a \emph{graph with edges labelled by $L$}.

We say a graph is \emph{2-vertex-connected} if it is connected and remains connected after deleting any one vertex. When we delete a vertex, we also delete all edges to it. The empty graph is not connected. 
\end{definition}

\begin{definition}
Let $L$ be a monoid, and $(H, \ell)$ a 2-vertex-connected graph labelled by $L$. We say $(H, \ell)$ is \emph{aligned} if for all pairs of edges $e$, $e'$ there exist positive integers $n$, $n'$ such that 
$$\ell(e)^{n} = \ell(e')^{n'}. $$ In other words, non-trivial relations should hold between the labels of edges in $H$. 

Let $L$ be a monoid, and $(G, \ell)$ a graph labelled by $L$. We say $G$ is \emph{aligned} if every 2-vertex connected subgraph\footnote{I.e. a subset $E' \sub E$ and $V' \sub V$ such that for all $e \in E'$, we have that $\on{ends}(e)$ lands in $V'$. } of $G$ is aligned (equivalently, if every circuit in $G$ is aligned). 
\end{definition}


\subsection{Local structure of semistable curves}

\begin{definition}
Let $k$ be a separably closed field. A \emph{curve over $k$} is a proper morphism $\pi\colon C \ra \on{Spec} k$ such that every irreducible component of $C/k$ is of dimension $1$. The curve $C/k$ is called \emph{semistable} if it is connected and, for every point $c \in C$, either $C \ra k$ is smooth at $c$, or $c$ has completed local ring isomorphic to $k[[x,y]]/(xy)$ (i.e. only ordinary double point singularities). 

Let $S$ be a scheme. A \emph{(semistable) curve over $S$} is a proper flat finitely presented morphism $C \ra S$, all of whose fibres over points with values in separably closed fields are (semistable) curves. 
\end{definition}

\begin{remark}\label{rem:cohom_flatness}
\begin{enumerate}
\item
In the classical definition of a semistable curve, one works with fibres over algebraically closed fields, not separably closed fields. This makes our definition appear more restrictive, but it is in fact equivalent to the classical definition, see \cite[10.3.7]{Liu2002Algebraic-geome}. 
\item Let $f\colon C \ra S$ be a semistable curve. Since $f$ is flat and has reduced geometric fibres, it is cohomologically flat in dimension 0 (\cite[7.8.6]{Grothendieck1963EGAIII.2} and reduction to the noetherian case). 
\end{enumerate}
\end{remark}

\begin{proposition}\label{local_rings_on_semistable}
Let $S$ be a locally noetherian scheme, $C/S$ a semistable curve, $s$ a geometric point of $S$, and $c$ a non-smooth geometric point of $C$ lying over $s$.  We have:
\begin{enumerate}
\item
There exists an element $\alpha$ in the maximal ideal of the completed \'etale local ring $ \widehat{\ca{O}\et_{S,s}}$ and an isomorphism of complete local rings
$$ \frac{ \widehat{{\ca{O}}\et}_{S,s}[[x,y]]}{(xy-\alpha)} \ra \widehat{\ca{O}\et}_{C,c} . $$
\item The element $\alpha$ is not in general unique, but the ideal $\alpha \widehat{\ca{O}\et_{S,s}} \triangleleft \widehat{\ca{O}\et_{S,s}}$ is unique. Moreover this ideal comes via base-change from a unique principal ideal of $\ca{O}\et_{S,s}$, which we call the \emph{singular ideal of $c$}.
\item Suppose moreover that $C$ is smooth over a schematically dense open $U \hra S$. Let $\alpha \in \ca{O}\et_{S,s}$ be a generator of the singular ideal. Then there is an isomorphism 
$$ \frac{ \widehat{{\ca{O}}\et}_{S,s}[[x,y]]}{(xy-\alpha)} \ra \widehat{\ca{O}\et}_{C,c} . $$
\end{enumerate}
\end{proposition}
\begin{proof}
Write $(A, \frak{m}_A) = \ca{O}\et_{S,s}$ and $(B, \frak{m}_B) = \ca{O}\et_{C,c}$, and $\rho:A \ra B$ for the canonical map. 

\begin{enumerate}
\item
From the definition of a semistable curve, we know that there is an isomorphism 
$$ \frac{(\hat{A}/\frak{m}_{\hat{A}}\hat{A})[[x,y]]}{(xy)}  \ra \hat{B}/\frak{m}_{\hat{A}}\hat{B} . $$
Let $u$, $v\in\hat{B}$ be any lifts of $x$ and $y$ respectively. We see that $uv \in \frak{m}_{\hat{A}} \hat{B}$, and that $\frak{m}_{\hat{B}} = u\hat{B} + v\hat{B} + \frak{m}_{\hat{A}}\hat{B}$. We now apply \cite[lemma 10.3.20]{Liu2002Algebraic-geome} to find an element $\alpha \in \frak{m}_{\hat{A}}$ and an isomorphism
$$\frac{\hat{A}[[x,y]]}{(xy-\alpha)} \ra \hat{B}  .$$ 
\item The uniqueness part is immediate from \cref{lem:unique_alpha}. For the `moreover' part, define 
$$F := \on{Fit}_{1}(\Omega^1_{B/A} / B). $$
The morphism $A \ra B/F$ is unramified and $B/F$ is local, so by \cite[\href{http://stacks.math.columbia.edu/tag/04GL}{Tag 04GL}]{stacks-project} we find that $B/F = A/I$ for some ideal $I \triangleleft A$, which will be the singular ideal. Write $\alpha \in\hat{A}$ for an element with 
$$ \frac{ \widehat{{\ca{O}}\et}_{S,s}[[x,y]]}{(xy-\alpha)} \stackrel{\sim}{\ra} \widehat{\ca{O}\et}_{C,c}. $$
I claim now that $I\hat{A} = \alpha\hat{A}$ as $\hat{A}$-modules. We establish the claim in two steps:
\begin{itemize}
\item[Step 1:] We have $\frac{\hat{B}}{F\hat{B}} = \frac{B}{F}\otimes_B \hat{B}$ which is the completion of the local ring $\frac{B}{F}$ with respect to its maximal ideal. Similarly $\frac{\hat{A}}{I\hat{A}} = \frac{A}{I}\otimes_A \hat{A}$ is the completion of the local ring $\frac{A}{I}$ with respect to its maximal ideal. Hence $\frac{\hat{B}}{F\hat{B}} = \frac{\hat{A}}{I\hat{A}}$ as $\hat{A}$-modules. 
\item[Step 2:] Formation of the relative differentials and of the Fitting ideal commutes with completions, so an easy calculation gives that $F \otimes_B \hat{B} = (x,y)\hat{B}$, and hence
\begin{equation*}
\frac{\hat{B}}{F\hat{B}} = \frac{\hat{B}}{(x,y)\hat{B}} = \frac{\hat{A}}{\alpha\hat{A}}
\end{equation*}
as $\hat{A}$-modules. Combining with step 1 we find that $\frac{\hat{A}}{I\hat{A}} = \frac{\hat{A}}{\alpha\hat{A}}$, and hence $I\hat{A} = \alpha\hat{A}$ as required. 
\end{itemize}
This shows that $I\hat{A} = \alpha\hat{A}$ is a principal ideal of $\hat{A}$. Since $\hat{A}$ is faithfully flat over $A$, we deduce by \cref{prop:descent_of_cartier} that $I$ is a principal ideal of $A$. The uniqueness comes again from faithful flatness.

\item We let $\alpha$ be as in part 1, and we choose $\alpha'$ to be a generator of the singular ideal $I$ defined in the proof of part 2. Then $\alpha$ and $\alpha'$ generate the same principal ideal of $\hat{A}$. By \cref{lem:non_zero_divisor} we have that $\alpha$ is not a zero-divisor, and so we deduce that $\alpha$ and $\alpha'$ differ by multiplication by a unit in $\hat{A}$. We can then easily write down an isomorphism
$$\frac{\hat{A}[[x,y]]}{(xy-\alpha)} \stackrel{\sim}{\ra} \frac{\hat{A}[[x,y]]}{(xy-\alpha')}, $$
so we get an isomorphism  
$$\frac{\hat{A}[[x,y]]}{(xy-\alpha')} \stackrel{\sim}{\ra} \hat{B}.$$ 
%
\end{enumerate}
\end{proof}
\begin{remark}
A similar argument shows that if $C/S$ is smooth over a schematically dense open, then given a point $s \in S$, a finite separable extension $k$ of the residue field of $s$, and $c$ a $k$-valued point of $\on{Sing}(C/S)$ lying over $s$, we may take the singular ideal of $c$ to be generated by an element in the unique finite \'etale extension of the henselisation of $\ca{O}_{S,s}$ whose residue field is $k$. 
\end{remark}

\begin{remark}
If $S$ is not assumed to be locally noetherian, one can define the singular ideal as the pullback of the fitting ideal of the sheaf of relative differentials (the ideal $F$ from the proof of \cref{local_rings_on_semistable}). One can then show that this ideal is principal by reducing to the noetherian case. However, most of our results will require that $S$ is even regular, so there seems little to be gained from such additional generality at this stage. 
\end{remark}

\begin{lemma}\label{lem:unique_alpha}
Let $(A, \frak{m}_A)$ be a local ring, $\alpha$, $\beta \in \frak{m}_A$ elements, and 
$$\phi: \frac{A[[x,y]]}{(xy-\alpha)} \ra \frac{A[[s,t]]}{(st-\beta)}$$
an isomorphism of $A$-algebras. Then  $\alpha A = \beta A$. 
\end{lemma}
\begin{proof}
Write 
$$R \colon = \frac{A[[x,y]]}{(xy-\alpha)} \;\; \text{ and }\;\; R' \colon = \frac{A[[s,t]]}{(st-\beta)}. $$
The first fitting ideal of the module of continuous relative differentials $\Omega_{R/A}$ is the ideal $(x,y) \triangleleft R$, and similarly
$$\on{Fit}_1(\Omega^1_{R'/A}/R') = (s,t) \triangleleft R'. $$ 
Since $R $ and $R'$ are by assumption isomorphic as $A$ modules, the same is true of the quotients
$$R/\on{Fit}_1(\Omega^1_{R/A}/R) \;\; \text{ and }\;\; R'/\on{Fit}_1(\Omega^1_{R'/A}/R'), $$
so we see that 
$$A/\alpha = R/(x,y) \;\; \text{ and }\;\; A/\beta = R'/(s,t)$$
are isomorphic as $A$-modules, hence $\alpha A = \beta A$. 
\end{proof}

\begin{lemma}\label{prop:descent_of_cartier}
Let $\phi\colon S' \ra S$ be a flat morphism of schemes, and $s' \in S'$ be a point lying over some $s \in S$. Let $D \tra S$ be a closed subscheme. Then the ideal sheaf of $D$ can be generated near $s$ by a single (regular) element if and only if the ideal sheaf of the pullback $D \times_S S' \tra S'$ can be generated near $s'$ by a single (regular) element
\end{lemma}
\begin{proof}
This is well-known and is omitted. 
\end{proof}


\begin{lemma}\label{lem:non_zero_divisor}
Let $\pi\colon C\ra S$ be semistable, and assume that $C \ra S$ is smooth over some schematically dense open subscheme $U \sub S$. Let $k$ be a separably closed field, and $c \in C(k)$ be a point lying over a point $s \in S(k)$ such that $c$ is not in the smooth locus of $C \ra S$. Suppose we have an element $\alpha \in \widehat{{\ca{O}}\et}_{S,s}$ and an isomorphism
$$ \frac{ \widehat{{\ca{O}}\et}_{S,s}[[x,y]]}{(xy-\alpha)} \ra \widehat{\ca{O}\et}_{C,c} . $$
Then $\alpha$ is a non zero-divisor in $\widehat{\ca{O}\et_{S,s}}$. 
\end{lemma}
\begin{proof}
We may assume that $S = \on{Spec}R $ with $R$ complete and strictly local; $U$ remains schematically dense by \cite[th\'eor\`eme 11.10.5 (ii)]{Grothendieck1966EGA.IV.3}. Let $Z$ be the closed subscheme of $S$ cut out by $\alpha$. From our smoothness assumption, it follows that $Z \times_S U$ is empty. 

Now we will assume that $\alpha$ is a zero-divisor in $R$, and derive a contradiction by showing that $Z \times_S U$ is \emph{non} empty. First, note that $Z$ contains an associated prime of $R$; indeed, the union of the associated primes of $R$ is exactly equal to the set of zero-divisors in $R$. However, I claim that $U$ must contain every associated prime of $R$. To see this, note that by \cite[\href{http://stacks.math.columbia.edu/tag/083P}{Tag 083P}]{stacks-project} $U$ is dense in $S$ and contains every embedded prime, hence it contains every associated prime (since every associated prime is either embedded, or is the generic point of an irreducible component). 
%
%
\end{proof}

\subsection{The definition of an aligned curve}


The most important definition in this paper is the following:
\begin{definition} \label{def:balance}
Let $S$ be a locally noetherian scheme and $C/S$ a semistable curve. Let $s \in S$ be a geometric point, and write $\ca{O}\et_{S,s}$ for the \'etale local ring of $S$ at $s$. Write $\Gamma$ for the dual graph\footnote{ Defined as in \cite[10.3.17]{Liu2002Algebraic-geome}. } of the fibre $C_s$.  Let $L_s$ be the monoid of principal ideals of $\ca{O}\et_{S,s}$. We label $\Gamma$ by elements of $L_s$ by assigning to an edge $e \in \Gamma$ the singular ideal $l \in L_s$ of the singular point of $C_s$ associated to $e$ (cf. \cref{local_rings_on_semistable}). 

We say \emph{$C/S$ is aligned at $s$} if and only if this labelled graph is aligned. We say \emph{$C/S$ is aligned} if it is aligned at $s$ for every geometric point $s$ of $S$. 
\end{definition}

\begin{remark}
\begin{enumerate}
\item Let $f\colon T \ra S$ be any morphism. Let $t$ be a geometric point of $T$, lying over a geometric point $s$ of $S$. Then the labelled graph $\Gamma_t$ of $C_T$ over $t$ has the same underlying graph as that of $C$ over $s$, and the labels on $\Gamma_t$ are obtained by pulling back those on $\Gamma_s$ along $f$. One can see this for example by using the construction of the singular ideal in terms of the fitting ideal of the sheaf of relative differentials, whose formation is well-behaved under base-change. 
\item
The property of `being aligned' is fppf-local on the target, i.e. it is preserved under flat base-change and satisfies fppf descent.  
\item
Given $S$ excellent, integral and regular in codimension 1, and $C/S$ semistable and generically smooth, one sees easily that there exists an open subscheme $V \sub S$ such that $C_V/V$ is aligned and such that the closed subscheme $S-V$ has codimension at least 2 in $S$.  
\end{enumerate}
\end{remark}

\subsection{Examples of aligned and non-aligned curves}\label{sec:basic_examples}
Let $S = \on{Spec}\bb{C}[[s,t]]$. We give two semistable curves $C_1$, $C_2$ over $S$, with the same closed fibre, and with $C_1$ aligned and $C_2$ non-aligned. 

\begin{example}[The aligned curve $C_1$]
We define $C_1$ to be the $S$-scheme cut out in weighted projective space $\bb{P}_S(1,1,2)$ (with affine coordinates $x$, $y$) by the equation
$$y^2 = ((x-1)^2-t)((x+1)^2+t). $$
This is naturally the pullback of the curve over $\on{Spec}\bb{C}[[t]]$ defined by the same equation inside $\bb{P}_{\bb{C}[[t]]}(1,1,2)$ along the natural map $S \ra \on{Spec}\bb{C}[[t]]$. The closed fibre is a 2-gon, and the labelled dual graph is 

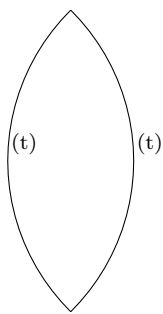
\begin{tikzpicture}[line cap=round,line join=round,>=triangle 45,x=1.0cm,y=1.0cm]
\clip(-9,0) rectangle (3.0,4);
\draw [shift={(0.0,2.0)}] plot[domain=2.356194490192345:3.9269908169872414,variable=\t]({1.0*2.8284271247461903*cos(\t r)+-0.0*2.8284271247461903*sin(\t r)},{0.0*2.8284271247461903*cos(\t r)+1.0*2.8284271247461903*sin(\t r)});
\draw [shift={(-4.0,2.0)}] plot[domain=-0.7853981633974483:0.7853981633974483,variable=\t]({1.0*2.8284271247461903*cos(\t r)+-0.0*2.8284271247461903*sin(\t r)},{0.0*2.8284271247461903*cos(\t r)+1.0*2.8284271247461903*sin(\t r)});
\begin{scriptsize}
\draw[color=black] (-2.609524887461456,2.2251621195229765) node {(t)};
\draw[color=black] (-0.9546805806813534,2.2251621195229765) node {(t)};
\end{scriptsize}
\end{tikzpicture}

The relation $(t) = (t)$ holds, implying that $C_1/S$ is aligned (there is only one circuit in the dual graph, with only two edges, so there is only one condition to check). 
\end{example}

\begin{example}[The non-aligned curve $C_2$]
We define $C_2$ to be the $S$-scheme cut out in weighted projective space $\bb{P}_S(1,1,2)$ (with affine coordinates $x$, $y$) by the equation
$$y^2 = ((x-1)^2-s)((x+1)^2+t). $$
This does not arise as the pullback of a curve over any trait. The closed fibre is a 2-gon, and the labelled dual graph is 

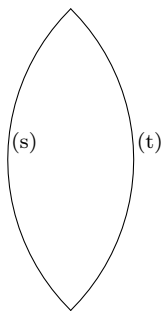
\begin{tikzpicture}[line cap=round,line join=round,>=triangle 45,x=1.0cm,y=1.0cm]
\clip(-9,0) rectangle (3.0598491265074137,4);
\draw [shift={(0.0,2.0)}] plot[domain=2.356194490192345:3.9269908169872414,variable=\t]({1.0*2.8284271247461903*cos(\t r)+-0.0*2.8284271247461903*sin(\t r)},{0.0*2.8284271247461903*cos(\t r)+1.0*2.8284271247461903*sin(\t r)});
\draw [shift={(-4.0,2.0)}] plot[domain=-0.7853981633974483:0.7853981633974483,variable=\t]({1.0*2.8284271247461903*cos(\t r)+-0.0*2.8284271247461903*sin(\t r)},{0.0*2.8284271247461903*cos(\t r)+1.0*2.8284271247461903*sin(\t r)});
\begin{scriptsize}
\draw[color=black] (-2.609524887461456,2.2251621195229765) node {(s)};
\draw[color=black] (-0.9546805806813534,2.2251621195229765) node {(t)};
\end{scriptsize}
\end{tikzpicture}

There do not exist positive integers $m$ and $n$ such that the relation $(s)^m = (t)^n$ holds. As $(s)$ and $(t)$ appear on the same circuit, this implies that the curve is not aligned. 
\end{example}

\begin{remark}
Let $C/S$ be a semistable curve, and let $s$ be a geometric point of $S$. Suppose firstly that the fibre over $s$ is
2-vertex-connected (for example an $n$-gon for some $n$). Then $C/S$ being
aligned at $s$ means roughly that $C/S$ looks (locally at $s$) as if it
is the pullback of a semistable curve over some trait (this is only a heuristic, and cannot be used as a substitute definition). Suppose on the other hand that $C/S$ is of compact type (i.e. the dual graphs of all geometric fibres are trees, or equivalently the jacobian $\on{Pic}^0_{C/S}$ is an abelian scheme); then $C/S$ is automatically aligned. In both of these situations, it seems perhaps plausible that
a N\'eron model can exist. In general, being aligned can be thought of as some kind of common
generalisation of these two situations.
\end{remark}

\subsection{Local-global for relations between divisors}

This section contains various lemmas we will need for proving that both being aligned and being non-aligned are preserved under alterations (\cref{prop:persistence}). 

\begin{lemma}\label{lem:noeth_path_connected}
Let $X$ be a connected noetherian scheme, and $p$, $q \in X$. Then there exist points $x_1, \cdots, x_n \in X$ with $x_1 = p$, $x_n = q$, and for all $i$, either $x_i \in \overline{\{x_{i+1}\}}$ or $x_{i+1} \in \overline{\{x_{i}\}}$. 
\end{lemma}
\begin{proof}
Since $X$ is noetherian it has finitely many irreducible components, so we are done by \cite[Tag 0904]{stacks-project}. 
\end{proof}

 We say a monoid $M$ is \emph{cyclic} if it can be generated by a single element. 

\begin{definition}
Let $\frak{S}$ denote the set of submonoids of the monoid\footnote{For us $\bb{N}_0 \defeq \{n \in \bb{Z} | n \ge 0\}. $ } $\bb{N}_0 \times \bb{N}_0$. Given an element $M \in \frak{S}$, we define the \emph{saturation} of $M$ to be the set of all $m \in \bb{N}_0 \times \bb{N}_0$ such that $am \in M$ for some $a \in \bb{N}_{>0}$. We remark that if $M$ is cyclic then so is its saturation. 

Given a scheme $X$ and two effective Cartier divisors $D$, $E$ on $X$, we define a function
\begin{equation}
\begin{split}
\zeta = \zeta_{D,E} \colon X & \ra \frak{S} \\
x & \mapsto \{(m,n) | mD = nE \text{ locally at } x\}. 
\end{split}
\end{equation}
Note that a relation holds Zariski-locally if and only if it holds fppf-locally. 
\end{definition}


\begin{lemma}\label{prop:inclusion}
If $x \in \overline{\{y\}}$ then $\zeta(x) \sub \zeta(y)$. 
\end{lemma}
\begin{proof}
Say $mD = nE$ holds locally at $x$. Then the same relation holds on some Zariski neighbourhood $U$ of $x$, and $U$ is also a Zariski neighbourhood of $y$. 
\end{proof}


\begin{lemma}\label{prop:cyclic}
Let $x \in X$ lie in the union of the supports of $D$ and $E$. Then the saturation of the monoid $\zeta(x)$ is cyclic. 
\end{lemma}
\begin{proof}
Suppose not, and let $(m_1, n_1)$ and $(m_2, n_2)$ be elements of $\zeta(x)$ which are not contained in the same cyclic submonoid. Then (working locally at $x$) we have that $m_1D = n_1E$ and $m_2D = n_2E$, hence $n_1m_2E = m_1n_2E$. 

Now let $f\in\ca{O}_{X,x}$ be a defining equation for $E$. Because $E$ is Cartier $f$ cannot be a zero divisor. Because $n_1m_2E = m_1n_2E$ we have that 
\begin{equation*}
(f^{n_1m_2}) = (f^{m_1n_2})
\end{equation*}
as ideals of $\ca{O}_{X,x}$, and because $f$ is not a zero-divisor we deduce that there exist $u \in \ca{O}_{X,x}^\times$ such that $f^{n_1m_2} = uf^{m_1n_2}$. Therefore $f^{\abs{n_1m_2 - m_1n_2}} \in \ca{O}_{X,x}^\times$, so either $n_1m_2 = m_1n_2$ or $E$ is 0 at $x$. 

Applying the same argument to $D$, we find that either $n_1m_2 = m_1n_2$ or $D$ is 0 at $x$. Since $x$ is in the union of the supports of $D$ and $E$ it must be that $n_1m_2 = m_1n_2$. It follows that the saturation of $\zeta(x)$ is cyclic. 
\end{proof}

\begin{lemma}\label{prop:same_saturation}
Let $x$, $y\in X$  and assume that 
\begin{enumerate}
\item
$x \in \overline{\{y\}}$; 
\item $x$ and $y$ both lie in the union of the supports of $D$ and $E$;
\item $\zeta(x)$ and $\zeta(y)$ are both unequal to $\{(0,0)\}$. 
\end{enumerate}
Then $\zeta(x)$ and $\zeta(y)$ have the same saturations. 
\end{lemma}
\begin{proof}
We have $\zeta(x) \sub \zeta(y)$ (\cref{prop:inclusion}), and so the same holds for the saturations. By \cref{prop:cyclic} both the saturations are cyclic and hence are equal. 
\end{proof}

Combining the above results, we obtain:
\begin{lemma}\label{local_global_for_relation}
Let $X$ be a noetherian scheme, and $D$, $E$ as above and such that $\on{Supp} D \cup \on{Supp} E$ is connected. Suppose that for all $x \in X$ we have that $\zeta(x) \neq \{(0,0)\}$ (i.e. non-trivial relations hold between $D$ and $E$ everywhere locally). Then there exist integers $m$, $n$ with $(m,n) \neq (0,0)$ and $mD = nE$ holds (globally on $X$). 
\end{lemma}
\begin{proof}
Let $S := \on{Supp} D \cup \on{Supp} E$. By quasi-compactness, we find 
\begin{enumerate}
\item
a finite open cover $U_i$ $(i = 1, \cdots, n)$ of $X$ with $U_i \cap S$ non-empty for all $i$; 
\item for each $i$, a point $p_i \in U_i \cap S$;
\item for each $i$, integers $m_i$, $n_i$ with $(m_i, n_i) \neq (0,0)$ such that for all $i$, $n_iD = m_iE$ holds on $U_i$. 
\end{enumerate}

 We want to show that some non-trivial relation holds between $D$ and $E$ on the whole of $X$. For any $i$ and $j$, we find a chain of specialisations joining $p_i$ and $p_j$ (by \cref{lem:noeth_path_connected}), and so the saturations of $\zeta(p_i)$ and $\zeta(p_j)$ coincide by \cref{prop:same_saturation}. 
 
 Now if a finite collection of cyclic monoids in $\bb{N}_0 \times \bb{N}_0$ all have the same saturation (and are not equal to $\{(0,0)\}$) then the monoids must contain a common non-zero element. This element is the relation we wanted.  
\end{proof}

\subsection{Descent of alignment along proper surjective maps}

The next theorem is the goal of this section; it shows that being aligned is preserved under and descends along proper surjective base change (e.g. under alterations). 
\begin{theorem}\label{prop:persistence}\label{lem:pull_back_balance}
Let $S$ be locally noetherian and normal, and let $C \ra S$ semistable and smooth over some schematically dense open $U \sub S$. Let $f\colon S' \ra S $ be a proper surjective map. Write $C'/S'$ for the pullback of $C$. Then $C/S$ is aligned if and only if $C'/S'$ is aligned. 
\end{theorem}
\begin{proof}
Suppose $C /S$ is aligned, and let $s'$ and $s$ be geometric points of $S'$ and $S$ respectively, with $s'$ lying over $s$. There is a natural isomorphism of graphs $\Gamma_{s} \ra \Gamma_{s'}$, and the labels on $\Gamma_{s'}$ are given by pulling back those on $\Gamma_s$. It is then easy to check that alignment of $C/S$ implies alignment of $C'/S'$.


Conversely, suppose $C'/S'$ is aligned. We want to show that the same holds for $C/S$. We may assume  that $S$ is strictly henselian local, with geometric closed point $s$. Shrinking $S'$, we may assume that $f^{-1}U$ is dense in $S'$. Fix a circuit $\gamma$ in the dual graph of the fibre over $s$, and two labelling divisors $D$ and $E$ that appear in that circuit - these are divisors since the labels are non-zero-divisors by \cref{lem:non_zero_divisor}. We will show that some non-trivial relation holds between $D$ and $E$. 

 Considering the Stein factorisation of $S' \ra S$, and the fact that $S$ is strictly henselian, we may assume that the fibre $S'_s$ is connected \cite[\href{http://stacks.math.columbia.edu/tag/03QH}{Tag 03QH}]{stacks-project}. Note also that every irreducible component of each of $f^{*}D$, $f^{*}E$ meets the closed fibre; indeed, let $Z$ be such an irreducible component, then $f(Z)$ is closed (since $f$ is proper) and so contains the closed point $s$ of $S$, hence $Z$ meets the closed fibre $f^{-1}(s)$. These facts together imply that the union $\on{Supp}f^{*}D \cup \on{Supp}f^{*}E$ of the supports of the pullbacks $f^{*}D$, $f^{*}E$ is also connected. 

Now by the assumption that $C'/S'$ is aligned we find that, locally at every point in $S'$, a non-trivial relation holds between $f^*D$ and $f^*E$. Indeed, since $S$ is local it is enough to check this at points of $S'$ lying over $s\in S$. Let $s'$ be such a point, with $\bar{s}'$ an associated geometric point. Then the (unlabelled)  graph of $C'/S'$ at $\bar{s}'$ is naturally identified with the (unlabelled) graph of $C$ over $s$, and the labels on the graph over $\bar{s}'$ are just given by pulling back. In particular, the restrictions of $f^*D$ and $f^*E$ to the \'etale local ring of $S'$ at $\bar{s}'$ correspond to two edges on the same circuit, and so a non-trivial multiplicative relation holds between them by the assumption that $C'/S'$ is aligned. This descends to a multiplicative relation on some \'etale neighbourhood of $\bar{s}'$ by a finite presentation argument. 

A-priori this does not imply that a non-trivial relation holds globally on $S'$, but using connectedness we can apply \cref{local_global_for_relation} to deduce that, for some $(m,n) \neq (0,0)$ we have $mf^*D = nf^*E$ globally on $S'$. By \cref{lem:injectivity_of_pullback} this implies that the same non-trivial relation holds on $S$: we have $mD = nE$. 
\end{proof}

We show that the pull-back map on Cartier divisors (\emph{not} up to linear equivalence) is injective:
\begin{lemma}\label{lem:injectivity_of_pullback}
Let $S$ be a locally noetherian normal scheme, and $f\colon S' \ra S$ a proper morphism whose scheme theoretic image is $S$ (i.e. $f$ does not factor via a non-trivial closed immersion). Let $U \sub S$ be a schematically dense open subscheme such that $f^{-1}U$ is schematically dense in $S'$. 

Let $D$ and $E$ be effective Cartier divisors on $S$, supported on $S \setminus U$ (so $f^*D$ and $f^*E$ are again Cartier divisors) and such that $f^*D = f^*E$. Then $D = E$.  
\end{lemma}
\begin{proof}
Because $S$ is normal, the effective Cartier divisors $D$ and $E$ are uniquely determined by the generic points (together with multiplicities) of their associated closed subschemes (\cite[21.7.2]{Grothendieck1967EGA.IV.4}). We may thus reduce to the case where $S$ is the spectrum of a discrete valuation ring $R$ with uniformiser $r$, and $D$ and $E$ are given by the vanishing of some powers of $r$, say $D = (r^n)$ and $E = (r^m)$. 

By the condition that $f^{-1}U$ is schematically dense in $S'$, we see that $f^*r$ is regular, and hence that $f^*r$ is regular in $A := \on{H}^0(S', \ca{O}_{S'})$. The condition that $f^*D = f^*E$ is then equivalent to the condition that $f^*r^n$ and $f^*r^m$ differ by multiplication by a unit in $A$. Let $a \in A$ be such a unit, so $a\cdot f^*r^m = f^*r^n$. Using that $f^*r$ is regular, we deduce that $a = f^*r^{n-m}$, and hence that either $m=n$ (in which case $D=E$ and we are done) or that $f^*r$ is a unit in $A$. However, $f^*r$ cannot be a unit since $f$ is surjective. 
%
%
%
%
%
\end{proof}

\begin{example}
Here we given an example with the normalisation of a nodal cubic to show that the assumption that $S$ be normal is necessary. Let $k$ be a field, let $A_0 = k[t]$, and let $R_0 = k[t^2, t^3 - t] \sub A_0$. Let $R$ and $A$ be the rings obtained from $R_0$ and $A_0$ by inverting $t^2$, and set $S = \on{Spec}R$ and $S' = \on{Spec}A$. Define Cartier divisors $D$ and $E$ on $S$ by $D = (t^3 -t)$ and $E = (t^2 - 1)$. Then $D$ and $E$ are not equal, but their pullbacks to $S'$ are equal since $t(t^2 - 1) = t^3 - t)$ and $t$ is a unit in $A$. 
\end{example}

\section{Changing the model of $C$}
\label{sec:changing_the_model}

\begin{definition}
Let $S$ be a scheme, and $U \sub S$ a schematically dense open subscheme. Let $C/U$ be a smooth proper curve. A \emph{model} of $C$ over $S$ is a proper flat curve $\bar{C}/S$ together with a $U$-isomorphism $(\bar{C})_U \ra C$. 
\end{definition}
From \cref{thm:intro_NM} we know that 
\begin{itemize}
\item
if $C$ has a semistable non-aligned model over $S$ then the jacobian $\on{Jac}_{C/U}$ has no N\'eron model over $S$;
\item if $C$ has a regular semistable aligned model over $S$ then the jacobian $\on{Jac}_{C/U}$ does have a N\'eron model over $S$. 
\end{itemize}
What can we say about the existence of a N\'eron model if we are presented with a semistable aligned model of $C$ which is not regular? We are unable to answer this question in complete generality as our method of producing N\'eron models for the jacobian requires a regular model of $C$, which is not known to exist. On the other hand, suppose that $U$ is the complement of a divisor with strict normal crossings and that $\bar{C}/S$ is split semistable in the sense of \cite[2.22]{Jong1996Smoothness-semi}\footnote{I.e. for all field-valued fibres of $\bar{C}/S$ we have that all the irreducible components are geometrically irreducible and smooth, and that all singular points are rational.}. Then we will see in this section how to determine from the labelled graphs of $\bar{C}/S$ whether or not the jacobian admits a N\'eron model. 

Before going the general theory, let us consider an example. Let $S = \on{Spec}\bb{C}[[u,v]]$, and define $\bar{C} \sub \bb{P}_S(1,1,2)$ by the affine equation 
\begin{equation*}
y^2 = ((x-1)^2-uv)((x+1)^2+uv). 
\end{equation*}
This is split semistable and smooth over the open subset $U := S \setminus (uv=0)$. The labelled graph of $\bar{C}$ at the closed point of $S$ is a 2-gon, and both edges are labelled by $(uv)$. As such, the graph is aligned, but clearly $\bar{C}$ is not regular. After some blowups we obtain a new model $\tilde{C}$ of $C := \bar{C}_U$, with labelled graph over the closed point of $S$ a 4-gon with two labels equal to $(u)$ and two labels equal to $(v)$. Clearly this is not aligned, and so the jacobian of $C$ does not admit a N\'eron model over $S$. 

To generalise this calculation, we need some definitions. 
\begin{definition}
Let $G$ be a set and $M(G)$ be the free (commutative) monoid on $G$. Given a graph $\Gamma$ with an edge-labelling by non-zero elements of $M(G)$, we say $\Gamma$ is \emph{regular} if every label lies in $G$. A \emph{refinement} of $\Gamma$ is a new $M(G)$-labelled graph $\Gamma'$ obtained from $\Gamma$ by replacing an edge with label $\prod_{g \in G} g^{\epsilon_g}$ by a chain of $\sum_{g \in G} \epsilon_g$ edges, with $\epsilon_g$ of those edges labelled by $g$ for every $g$ in $G$. 
\end{definition}
We do not specify the order of the labels on the new edges, so refinements are in general very non-unique. It is clear that every graph $\Gamma$ with finitely many edges becomes regular after a finite sequence of refinement operations. 
\begin{definition}
We say an $M(G)$-labelled graph $\Gamma'$ is \emph{a regularisation} of $\Gamma$ if $\Gamma'$ is regular and can be obtained from $\Gamma$ by a finite sequence of refinements. 
\end{definition}
\begin{definition}
Let $S$ be a regular scheme, and $C/S$ a semistable curve. Let $s$ be a geometric point of $S$. Let $\Gamma_s$ be the labelled graph of $C$ over $s$, so its labels lie in $\widehat{\ca{O}_{S,s}^{\on{et}}}/\left(\widehat{\ca{O}_{S,s}^{\on{et}}}\right)^\times$, which is the free abelian monoid on the height 1 prime ideals of $\ca{O}_{S,s}^{\on{et}}$ (since $S$ is regular). We say $C/S$ is \emph{strictly aligned at $s$} if one (equivalently all) regularisations of $\Gamma_s$ are aligned. We say $C/S$ is \emph{strictly aligned} if it is so at every geometric point of $S$. 
\end{definition}
\begin{remark}
\begin{enumerate}
\item If $C/S$ is strictly aligned then it is aligned, by \cref{thm:intro_NM}. 
\item
If $C$ is regular then all the graphs $\Gamma_s$ are automatically regular, so that $C/S$ is aligned if and only if it is strictly aligned. 
\end{enumerate}
\end{remark}

\begin{proposition}\label{prop:}
Let $S$ be a regular excellent separated scheme, let $U \sub S$ be a dense open subscheme which is the complement of a strict normal crossings divisor. Let $C/S$ be a split semistable curve which is smooth over $U$. Then the jacobian of $C_U/U$ has a N\'eron model over $S$ if and only if $C/S$ is strictly aligned. 
\end{proposition}
\begin{proof}
By \cite[3.6]{Jong1996Smoothness-semi} we know that $C/S$ admits a resolution of singularities; more precisely, there exists a proper $S$-morphism $\tilde{C} \ra C$ which is an isomorphism over $U$ and such that $\tilde{C}$ is split semistable and regular. Following through the steps in de Jong's proof, we find that for all geometric points $s$ of $S$, the labelled graph of $\tilde{C}$ over $s$ is a regularisation of the labelled graph of $C$ over $s$. 

Now if $C/S$ were strictly aligned, then by definition these regularisations are aligned, so a N\'eron model exists by \cref{thm:intro_NM}. Conversely, suppose that a N\'eron model exists. Then again by \cref{thm:intro_NM} we have that $\tilde{C}/S$ is aligned, so $C/S$ is strictly aligned. 
\end{proof}

\section{Classification of vertical Cartier divisors on certain complete local rings}\label{sec:Cartier_classification}
Let $(R, \frak{m}_R)$ be a regular $\frak{m}_R$-adically complete local ring. Let $r \in R$ be a non-zero element.  Let $A = R[[x,y]]/(xy-r)$. Then $A$ is $R$-flat and is a complete normal local noetherian domain (\cite[7.8.3, page 215]{Grothendieck1965Elements-de-geo}). Our aim is to classify the principal ideals of $A$ which become trivial after base-change over $R$ to $K \defeq \on{Frac} R$. More precisely, we will show:
\begin{theorem}\label{CDthm:main}
Let $a \in A$ be an element such that $a \otimes 1$ is a unit in $A \otimes_R K$. Then there exist 
\begin{itemize}
\item an element $s \in R$;
\item non-negative integers $m$, $n$ such that $mn = 0$;
\item a unit $u \in A^\times$;
\end{itemize}
such that $a = sx^ny^mu$. 
\end{theorem}
\noindent The proof will occupy the remainder of this section.  This result is only relevant for showing that the jacobians of non-aligned curves do not admit N\'eron models; is not required for the implication `aligned $\implies$ existence of N\'eron model'.

 The author is very grateful to the second referee for providing a new proof, based on the original one but half as long and far more readable. 


We begin with some lemmas to reduce to the case where the element $a$ is a polynomial. Write $\tilde{A}$ for the localisation of $\frac{R[x,y]}{(xy-r)}$ at the maximal ideal $(x,y) + \frak{m}_R$, so that $A$ is the completion of $\tilde{A}$ at its maximal ideal, and write $\phi:\tilde{A} \ra A$ for the canonical injection. Note that $\tilde{A}$ is again normal, again by \cite[4.21]{Jong1996Smoothness-semi}. Write $\on{Div}A$ for the free group generated by height 1 prime ideals of $A$ (Weil divisors of $\on{Spec} A$), and write $\on{Div}_FA$ for the subgroup generated by height 1 primes $\frak{p}$ such that $\frak{p}\otimes_R K \cong A \otimes_R K$ (`fibral primes'). We define $\on{Div}\tilde{A}$ and $\on{Div}_F\tilde{A}$ analogously. 

\begin{lemma}\label{lem:same_Weil_divisors}
The map $\phi^*:\on{Div}_F {A} \ra \on{Div}_F \tilde{A}$ is well-defined and an isomorphism. 
\end{lemma}
\begin{proof}
Fibral primes on $A$ and on $\tilde{A}$ are exactly those height 1 primes whose preimages in $R$ are also of height 1 (for dimension reasons). From this we deduce easily that the map is well defined. Since $\tilde{A}$ is regular in codimension 1, we deduce by faithful flatness of the completion that $\phi^*$ is injective. It remains to show that it is surjective. 

For this, it is enough to show that $\phi^*$ induces a surjection on the fibral primes. Note that if a fibral prime $p$ of $A$ lies over a prime $q$ of $R$ then the same holds for $\phi^*p$. Now let $q$ be a height 1 prime ideal of $R$. If $r \notin q$ then $qA$ and $q\tilde{A}$ are fibral prime ideals of height 1. If $r \in q$ then there are exactly two (fibral) height 1 prime ideals of $A$ above $q$, namely $q + xA$ and $q + yA$, and similarly for $\tilde{A}$. In either case, restricted to fibral primes lying over $q$, the map $\phi^*$ is a injective map of finite sets of the same cardinality, and so is also surjective. 
\end{proof}

\begin{lemma}\label{cor:finite_width_ok}
Let $a \in A$ be an element such that $a \otimes 1$ is a unit in $A \otimes_R K$. Then there exists $\tilde{a} \in \tilde{A}$ and a unit $u \in A^\times$ such that $a = u\tilde{a}$. 
\end{lemma}
\begin{proof}
By \cref{lem:same_Weil_divisors} we find that there is a Weil divisor $D$ on $\on{Spec}\tilde{A}$ such that $\phi^*D = \on{div}_A a$. Identify $D$ with its associated closed subscheme of $S$ as in \cite[21.7.1]{Grothendieck1967EGA.IV.4}. Then $D$ satisfies the condition $(S_1)$ (see [loc. cit. 21.7.2]), and the same holds for $\phi^*D$ by \cite[7.8.3, page 215]{Grothendieck1965Elements-de-geo}. Moreover, $\phi^*D$ is equal to the closed subscheme $V(a) \tra S$ cut out by $a$ at all points of $S$ of codimension 1 (using that $S$ is normal and thus regular in codimension 1). Moreover, $V(a)$ satisfies $(S_1)$ (again using that $S$ is normal), so applying \cite[21.7.2]{Grothendieck1967EGA.IV.4} again we find that $\phi^*D = V(a)$ as closed subschemes of $S$. In particular $\phi^*D$ is generated by a single element, so by flat descent the same holds for $D$, say $D = V(\tilde{a})$. Then $a = u\tilde{a}$ for some unit $u \in \tilde{A}$ as required. 
\end{proof}

\begin{proof}[Proof of \cref{CDthm:main}]
The proof is divided into several steps. 

\begin{itemize}
\item[\textbf{Step 1:}] From \cref{cor:finite_width_ok} we easily see that there exists $\tilde{a} \in R[x,y]$ and a unit $u \in A^\times$ such that $a = u\tilde{a}$. As such, we may and will assume that $a$ is contained in the image of $R[x,y]$ in $A$. 

\item[\textbf{Step 2:}] Let $W$ denote the set of series $\sum_{i \in \bb{Z}} w_iT^i$ where $w_i \in K$ and $T$ is a formal variable. The set $W$ is naturally an abelian group under addition of series, but does not have a natural ring structure. It is a module over the ring $K[T]$, with action the usual multiplication. 

Let
$$F^+ := \left\{\sum_{i \in \bb{Z}}a_iT^i : \forall i \ge 0, a_i \in R \text{ and } , a_{-i}\in r^iR\right\}. $$

This set $F^+$ has a natural multiplication since $R$ is complete with respect to its maximal ideal and hence is also $r$-adically complete. In what follows we will identify $A$ with $F^+$ via the $R$-algebra isomorphism
\begin{equation*}
A =\frac{R[[x,y]]}{(xy-r)} \ra F^+; \; x \mapsto T, \;y \mapsto r/T. 
\end{equation*}
If $\frak{p}$ is a height 1 prime ideal of $R$ we write $v_\frak{p}$ for the normalised valuation associated with $\frak{p}$. If $v_\frak{p}(r) >0$ then we define $A(\frak{p})$, $W(\frak{p})$ and $F^+(\frak{p})$ in the same way as $A$, $W$ and $F^+$ but replacing $R$ by the completed localisation $\hat{R}_\frak{p}$. We obtain a commutative diagram of $R$-algebras

\[
\xymatrix{ A =\frac{R[[x,y]]}{(xy-r)} \ar[d] \ar[r]^{\sim} & F^+ \ar[d] \\
A(\frak{p}) =\frac{\hat{R}_\frak{p}[[x,y]]}{(xy-r)}  \ar[r]^{\sim}& F^+(\frak{p})}
\]


Let $\Omega$ be the set of height 1 prime ideals of $R$, and $\Omega_r = \{\frak{p} \in \Omega : v_\frak{p}(r) >0\}$. 

\item[\textbf{Step 3:}] By step 1, it is enough to consider an element $a \in R[T, r/T]\sub F^+ = A$ (here we identify $F^+$ with $A$ as in step 2) and show it can be written as $a = sT^nu$ with $s \in R$, $n \in \bb{Z}$ and $u \in A^\times$. To prove this, we may and do assume that $a$ is of the form 
\begin{equation*}
a = \sum_{i=0}^ma_iT^i\in R[T] \sub F^+, \text{ with } a_0a_m \neq 0 \text{ and } \on{gcd}(a_0, \cdots, a_m) = 1\footnote{Since $R$ is regular it is a UFD, and so the $\on{gcd}$ makes sense. }. 
\end{equation*}

For each $\frak{p} \in \Omega_r$, let $\on{NP}(a;\frak{p})$ be the Newton polygon of $a$ with respect to $v_\frak{p}$ (here $a$ is viewed as an element in $\hat{R}_\frak{p}[[T]] \sub F^+(\frak{p})$). We claim that for every $\frak{p} \in \Omega_r$, $\on{NP}(a;\frak{p})$ does not contain any side of slope in $(-v_\frak{p}(r),0)$. Otherwise, pick a $\frak{p} \in \Omega_r$ such that the latter assertion does not hold.  The theory of Newton polygons (cf. \cite{Artin1967Algebraic-numbe}, especially the last paragraph of page 43) tells us that the polynomial $a = a(T)$ has a root $x_0 \in \overline{K}_\frak{p}$ with $\frak{p}$-adic valuation in $(0,v_\frak{p}(r))$ (here ${K}_\frak{p}$ denotes the completion of $K$ with respect to the valuation $v_\frak{p}$, and $\overline{K}_\frak{p}$ its algebraic closure). Then $y_0 \coloneqq r/x_0 \in \overline{K}_\frak{p}$ has $\frak{p}$-adic valuation $> 0$, and so the pair $(x_0, y_0)$ defines a $\bar{K}_\frak{p}$-valued point of the closed subscheme $V(a \otimes 1) \tra \on{Spec}(A(\frak{p})\otimes_{\hat{R}_\frak{p}}K_p)$. It follows that the closed subscheme $V(a\otimes 1) \tra \on{Spec}(A\otimes_R K)$ is not empty, contradicting our assumption that $a \otimes 1$ is a unit in $A \otimes_R K$. 

\item[\textbf{Step 4:}] For each $\frak{p} \in \Omega_r$, let $N_\frak{p}$ denote the smallest integer in $[0,m]$ such that $v_\frak{p}(a_{N_\frak{p}}) = 0$. Assume for the moment that integer $N_\frak{p}$ does not depend on $\frak{p} \in \Omega_r$ (we shall denote this integer by $N$), and that $a_N \in R^\times$. Under these extra assumptions, we claim that $a = T^N\cdot(\text{unit})$. Indeed, by step 3, we have that $v_\frak{p}(a_{N-i}) \ge i\cdot v_\frak{p}(r)$ for every $\frak{p} \in \Omega$ and $0 \le i \le N$. Hence $a_{N-i}/r^i \in R$ for $0 \le i \le N$. Write $a_{N-i} = r^ia'_{N-i}$. We find
\begin{equation*}
a = \sum_{i=0}^Nr^ia'_{N-i}T^i + \sum_{j > N}a_jT^j = x^N \left( a_N + \sum_{i=1}^N a'_{N-i} y^i + \sum_{j > N} a_j x^{j-N}\right). 
\end{equation*}
Since $a_N \in R^\times$, we see that $u := a_N + \sum_{i=1}^N a'_{N-i} y^i + \sum_{j > N} a_j x^{j-N} \in A^\times$ as required. 

\item[\textbf{Step 5:}]
It remains to show that, under the assumption that $a \otimes 1 \in (A \otimes_R K)^\times$, the extra assumptions in step 4 are always satisfied. If not, there are two possibilities:
\begin{itemize}
\item there exist $\frak{p}$, $\frak{q}\in \Omega_r$ such that $N_\frak{p} \neq N_\frak{q}$. Without loss of generality we assume $i_0 := N_\frak{q} > N_\frak{p} =: N$. In particular, $0 = v_\frak{q}(a_{i_0}) < v_\frak{q}(a_N)$. 
\item all the $N_\frak{p}$ are equal but $a_N \notin R^\times$. In this case we choose $\frak{p} \in \Omega_r$ and set $N = N_\frak{p}$. Let $\frak{q} \in \Omega$ be such that $v_\frak{q}(a_N) \neq 0$. As $\on{gcd}(a_0, \cdots, a_m) = 1$ there exists $i_0 \neq N$ such that $v_\frak{q}(a_{i_0}) = 0$. Up to the reflection substitution $T \mapsto r/T$ we may assume that $i_0 >N$. 
\end{itemize}
In both cases, we get $\frak{p} \in \Omega_r$, $\frak{q} \in \Omega$ and $i_0 > N = N_\frak{p}$ such that $v_\frak{q}(a_N)> v_\frak{q}(a_{i_0}) = 0$. Furthermore, we may assume $v_\frak{p}(a_{i_0}) = 0$. Indeed, if $v_\frak{p}(a_{i_0})>0$ consider the element 
\begin{equation*}
a  \left( 1 + a_N T^{i_0 - N}\right) =: \sum_{j=0}^{m+1}b_jT^j. 
\end{equation*}
One checks that for all $1 \le i \le N$ we have $v_\frak{p}(b_{N-i}) \ge i\cdot v_\frak{p}(r)$ and $v_\frak{p}(b_N) = 0$. Further, $b_{i_0} = a_{i_0} + a_N^2$. Therefore $v_\frak{p}(b_{i_0}) = 0$ since $a_N$ is a $\frak{p}$-unit, and $v_\frak{q}(b_{i_0}) = 0$ since $v_\frak{q}(a_{i_0})= 0$ and $v_\frak{q}(a_N) > 0$. Finally, we still have $v_\frak{q}(b_N)>0$ since $b_N$ is either $a_N$ (if $i_0 - N >N$) or $a_N(1+a_{2N-i_0})$ (if $i_0 - N \le N)$. So up to replacing $a$ by $a  \left( 1 + a_N T^{i_0 - N}\right) $ we may always assume that $a_{i_0}$ is a $\frak{p}$-unit. 

Now we claim that such an element $a$ is not invertible in $A\otimes K$. Otherwise, denote its inverse in $A\otimes K$ by $a^{-1}$. Define $f \in A(\frak{p})$ by 
\begin{equation*}
-f = \frac{a_0}{a_N} T^{-N} + \cdots + \frac{a_{N-1}}{a_N} T^{-1} + \frac{a_{N+1}}{a_N} T + \cdots + \frac{a_m}{a_N} T^{m-N}. 
\end{equation*}
One checks easily (cf. step 4) that the series $\sum_{i \ge 0}f^i$ converges to an element of $A(\frak{p})$, and that in $A(\frak{p}) \otimes K_\frak{p}$ we have 
\begin{equation*}
a a_N^{-1}T^{-N}\sum_{i \ge 0}f^i = 1. 
\end{equation*}
Now $A(\frak{p}) \otimes K_\frak{p}$ is a ring, so the image of $a^{-1}$ in $A(\frak{p}) \otimes K_\frak{p}$ must be equal to $ a_N^{-1}T^{-N}\sum_{i \ge 0}f^i$. 

Write $f = f_1 + f_2 \in \hat{R}_\frak{p}[r/T, T]$ where $f_1$ consists of those terms of $f$ which have both positive $T$-degree and coefficients which are $\frak{p}$-units. Since $v_\frak{p}(a_{i_0}) = 0$ we see $f_1 \neq 0$. Since $a^{-1} \in A\otimes K$, there exists $\lambda \in R$ such that $\lambda a^{-1} \in A$. Write $\lambda a^{-1} = \sum_{j \in \bb{Z}}\alpha_j T^j \in F^+ = A$. So in $A(\frak{p}) \otimes K_\frak{p}$ we find the equality
\begin{equation*}
\lambda\left( a_N^{-1}T^{-N}\sum_{i \ge 0}f^i\right) = \lambda\left( a_N^{-1}T^{-N}\sum_{i \ge 0}(f_1 + f_2)^i\right) = \sum_j \alpha_j T^j, 
\end{equation*}
and hence
\begin{equation*}
\lambda \sum_{i \ge 0} f^i = \lambda \sum_{i \ge 0}(f_1 + f_2)^i = a_N\sum_{j \in \bb{Z}} \alpha_jT^{j+N}. 
\end{equation*}

Since $a_N$ is a $\frak{p}$-unit, the series $\sum_{i\ge 0}f^i \in W(\frak{p})$ has coefficients in $\hat{R}_\frak{p}$, and $v_\frak{p}(\alpha_j) = v_\frak{p}(a_N\alpha_j) \ge v_\frak{p}(\lambda)$. Thus we may assume that  $\lambda \in R$ is a $\frak{p}$-unit. Write $\sum_{i \ge 0}f_1^i \eqqcolon \sum_{j \ge 0}\beta_jT^j \in 1 + TR[1/a_N][[T]]$. As all the coefficients of $f_2$ have $\frak{p}$-adic valuation strictly positive, we obtain
\begin{equation*}
v_\frak{p}(\alpha_{j-N}) > 0 \text{ for }j<0, \text{ and } v_\frak{p}(a_N\alpha_{j-N}-\lambda\beta_j) > 0 \text{ for } j \ge 0. 
\end{equation*}
Let $R'$ be the normalisation of the quotient $R/\frak{p}$, and set $K' = \on{Frac}(R')$, $A' = \frac{R'[[x,y]]}{(xy-\bar{r})}$. For any $b \in R[1/a_N]$, write $\bar{b}$ for its image in $K'$ - note that this makes sense since $a_N \notin \frak{p}$. In $K'[[T]]$ we find 
\begin{equation*}
\bar{\lambda} \sum_{i \ge 0} \bar{f}_1^i = \bar{a}_N\sum_{j \ge 0} \bar{\alpha}_{j-N}T^j \in R'[[T]]. 
\end{equation*}
We are done if we can show that an element of $K'[[T]]$ of the form $\bar{\lambda} \sum_{i \ge 0} \bar{f}_1^i$ can never be contained in $R'[[T]] \sub K'[[T]]$. By localising and completing at some generic point of the codimension-1 subscheme $V(\frak{q}R')\tra \on{Spec}(R')$, we deduce what we want from \cref{lem:liminf} applied to the polynomial $\bar{a}_N + \bar{a}_{N+1}T + \cdots + \bar{a}_mT^{m-N}$. 
\end{itemize} 
\end{proof}

\begin{lemma}\label{lem:liminf}
Let $R$ be a complete discrete valuation ring with fraction field $K$, valuation $v$ and uniformiser $\pi$. Let $f(T) = \sum_{i=0}^na_iT^i \in R[T]$ be a polynomial with $a_0 \neq 0$. Assume that the Newton polygon of $f$ has a side of slope $<0$ (in particular, that $f$ has positive degree). Write $f = a_0(1-\tilde{f})$ with $\tilde{f} \in TK[T]$, and 
\begin{equation*}
\sum_{i \ge 0}\tilde{f}^i \defeq \sum_{j \ge 0}b_jT^j \in K[[T]]. 
\end{equation*}
Then $\on{liminf}_j v(b_j) = -\infty$. 
\end{lemma}
\begin{proof}
Up to replacing $K$ by a finite extension, we may assume that $f$ splits in $K$, so all the slopes of the Newton polygon of $f$ are integers. Write $g(T) = f(\pi^\rho T)$ where $-\rho < 0$ denotes the smallest slope of the Newton polygon of $f$. Then the smallest slope of the Newton polygon of $g$ is $0$. We need to show that, if we write $\tilde{g} = \tilde{f}(\pi^\rho T) \in R[T]$, then there are infinitely many coefficients in $\sum_{i \ge 0}\tilde{g}$ which are invertible. Observe that $g = a_0(1-\tilde{g})$ and that $g$ has a side of slope zero with one end over $T^0$. Write $k$ for the residue field of $R$, and $h$ for the image of $\tilde{g}$ in $k[T]$. It is enough to show that the series $\sum_{i \ge 0} h^i \in k[[T]]$ is not a polynomial in $T$. Equivalently, we need to show that the inverse of $1-h$ in $k[[T]]$ is not a polynomial. This is clear since $\on{deg}(h)>0$ (since the Newton polygon of $g$ has a side of slope zero starting from $T^0$). 
\end{proof}


\section{Alignment is equivalent to flatness of closure of unit section in the relative Picard space}\label{sec:balance_equivalent_to_flatness}
In this section we return to our usual situation of a generically-smooth semistable curve $C$ over a base $S$. We want to understand more about the closure of the unit section in the relative Picard scheme of $C/S$. In fact, we will show that the closure of the unit section is flat (even \'etale) over $S$ if and only if $C/S$ is aligned.  In \cref{sec:Neron-models} we will relate the flatness of the closure of the unit section to the existence of N\'eron models.

\subsection{Test curves}

In this preliminary section, we will define `test curves' in $S$ which we will later use to detect flatness. 

\begin{definition}
Given $S$ a scheme, $s \in S$ and $U \sub S$ an open subscheme, a \emph{non-degenerate trait in $S$ through $s$} is a morphism $\phi:X \ra S$ where $X$ is the spectrum of a discrete valuation ring, and such that $\phi$ maps the closed point of $X$ to $s$ and the generic point of $X$ to a point in $U$. 
\end{definition}

\begin{lemma}[Non-degenerate traits exist]\label{lem:non_degenerate_traits_exist}
Let $S$ be a noetherian scheme, $s \in S$ a point, and $U\sub S$ a dense open subscheme. Then there exists a non-degenerate trait $X$ in $S$ through $s$. 
\end{lemma}
\begin{proof}
This is a special case of \cite[7.1.9]{Grothendieck1961EGAII}. 
%
\end{proof}

\subsection{A natural map to the closure of the unit section}\label{sec:construct_section}

The point of this section is to use the data of an integer weighting of the vertices of the reduction graph $\Gamma_{s}$ to define a map from $S$ to the closure of the unit section in the relative Picard scheme $\on{Pic}_{C/S}^{[0]}$, in the case where $C/S$ is aligned. 


\begin{definition}\label{def:vertex_labellings}
Let $(S,s)$ be a strictly henselian local scheme, and let $C \ra S$ be semistable and smooth over some schematically dense open subscheme $U \sub S$ (so that all labels of the graph $\Gamma_s$ are non-zero-divisors in $\ca{O}_S(S)$ by \cref{lem:non_zero_divisor}). An \emph{(integer) vertex labelling of the graph $\Gamma_s = (V,E,\on{ends})$} is a function $m:V \ra \bb{Z}$.

Let $\phi:T \ra S$ be a non-degenerate trait through $s$. We say a vertex labelling $m$ is \emph{$T$-Cartier} (better: $\phi$-Cartier) if for every edge $e$ in $\Gamma_s$ with endpoints $v_1$ and $v_2$, we have that $m(v_1)-m(v_2)$ is divisible by the thickness (see \cite[10.3.23]{Liu2002Algebraic-geome}) of the singular point in the special fibre of $C_T$ corresponding to $e$. 

Given a fibral Weil divisor $D$ on $C_T$, we define a vertex labelling $m$ of $\Gamma_s$ by attaching to a vertex $v$ the multiplicity of $D$ along the component corresponding to $v$. 
 \end{definition}

\begin{lemma}\label{lem:T_cartier_labelling}
In the above notation, let $m$ be a vertex labelling. Then $m$ is $T$-Cartier if and only if there exists a fibral Cartier divisor $D$ on $C_T$ whose associated labelling is $m$. 
\end{lemma}
\begin{proof}
This is \cite[page 15]{Raynaud1991Jacobienne-des-}. 
\end{proof}

If we suppose that $C/S$ is aligned at $s$, then from a $T$-Cartier vertex labelling $m$ we will construct a map from $S$ to the closure of the unit section in the Picard space $\Pictdz_{C/S}$. To give such a map is to give a line bundle on $C/S$, trivial over $U$. To construct this, we will first construct a Weil divisor on $C/S$, then prove that it is Cartier and hence defines a line bundle. 

When the base is a trait and $C$ is regular, one often constructs Cartier divisors on $C$ by taking irreducible components of the special fibre. If the base is a trait and $C$ is not regular but is semistable, then a similar procedure can be performed  taking care of multiplicities - see eg. \cite{Edixhoven1998On-Neron-models}. In our situation, the base $S$ is regular (but maybe of dimension greater than 1) and $C$ is semistable (but perhaps not regular). We will now carry out the analogous construction of Cartier divisors. 

\begin{definition}\label{definition_of_some_divisors}
We retain the above notation, and write $S = \on{Spec}A$. Let $a \in A$ a non-zero-divisor, non-unit. Let $E(a)$ denote the set of edges of $\Gamma_s$ whose labels $b$ have the property that the ideal $aA$ is a power of the ideal $bA$. Let $\Gamma_s(a)$ denote the graph obtained from $\Gamma_s$ be deleting all the edges in $E(a)$ (note that $\Gamma_s$ and $\Gamma_s(a)$ have the same vertex sets). 
\end{definition}
Suppose also that $S$ is regular. We will now give a recipe to associate a Cartier divisor on $C$ to any connected component of $\Gamma_s(a)$. We will then use these Cartier divisors to construct line bundles which will play an essential role in the proof of our main results. First, we need the notion of a `specialisation map' on labelled graphs. The construction of this map is carried in detail (and also much greater generality) in \cite[\S 5]{Holmes2014A-Neron-model-o}, so we only give an outline here. 

Recall that $S$ is strictly henselian local. Let $p \in S$ be any point. Then all singular points on the fibre $C_p$ are rational points, and moreover all irreducible components of $C_p$ are geometrically irreducible. To prove this, we note that $\on{Sing}(C/S) \ra S$ is finite and unramified, and hence is a disjoint union of closed immersions by \cite[\href{http://stacks.math.columbia.edu/tag/04GL}{Tag 04GL}]{stacks-project}, so the assertion on the singular points holds. For the geometric irreducibility, we note that there exists a section through the smooth locus of every irreducible component. 

Because of this, we naturally obtain a dual graph $\Gamma_p$ of $C$, even though the residue field of $p$ may not be separably closed. Similarly, we find that the labels on the graph $\Gamma_p$ can be defined in an analogous way to that used in the usual case, and moreover that these labels can be taken to be principal ideals of the Zariski local ring of $S$ at $p$ (whereas usually they live in the \'etale local ring). 

Given two points $p$, $q$ of $S$ with $p \in \overline{\{q\}}$, we have a canonical injective map
\begin{equation*}
\on{sp}\colon \ca{O}_{S,p} \ra \ca{O}_{S,q}. 
\end{equation*}
If we replace each label on the graph $\Gamma_p$ by its image under the map $\on{sp}$, then we obtain a new graph with labels in $\ca{O}_{S,q}$. Some of the labels may be units; contract all such edges. Then the resulting graph is naturally isomorphic to the labelled graph $\Gamma_q$. 

Write $\eta_1, \cdots, \eta_n$ for the generic points of $\on{Spec}A/a$ - these are height 1 prime ideals in $S$. Write $m_i$ for the order of vanishing of $a$ at $\eta_i$. Let $H$ be a connected component of the graph $\Gamma_s(a)$. Then for each $1 \le i \le n$, let $Z_i^1, \cdots, Z_i^{r_i}$ denote the vertices of $\Gamma_{\eta_i}$ which are images under the specialisation map of vertices in $H$. We can view each $Z_i^j$ as a prime Weil divisor on $C$. 

\begin{definition} Define a Weil divisor $\on{div}(a;H)$ by
\begin{equation*}
\on{div}(a;H) = \sum_{i=1}^n m_i \sum_{j=1}^{r_i} Z_i^j. 
\end{equation*}
We call such divisors \emph{primitive fibral Cartier divisors}; the name will be justified by the following lemma. 
\end{definition}

%

\begin{lemma}\label{lem:construct_cartier}
In the above setup, we have that $\on{div}(a;H)$ is a Cartier divisor on $C$. 
\end{lemma}
\begin{proof}
We may assume $S$ is complete. Since $C$ is noetherian, integral and normal it is enough by \cref{lem:check_on_closed} to check that $D$ is Cartier at closed points of $C$. Since $C \ra S$ is proper it maps closed points to closed points, and so it is enough to check that $D$ is Cartier at all closed points of the closed fibre of $C \ra S$. 

The result is clear at smooth points in the fibre (since smooth over regular implies regular), so it is enough to look at points of $C$ corresponding to edges in $\Gamma_s$. Let $e$ be an edge of $\Gamma_s$. If $e$ has exactly 0 or 2 endpoints in $H$ (for example, if $e$ is not in $E(a)$) then this is easy; in the first case the divisor $\on{div}(a;H)$ is cut out near the point corresponding to $e$ by a unit, and in the second case by the function $a$. To see this second case, we first reduce (by looking at the local ring on $S$ at one of the $\eta_i$) to the case where $S$ is a DVR, $\eta_i$ is the closed point, and $a$ is a power of a uniformiser $\pi$, so $a = \pi^{m_i}$. Let $e$ be an edge with both endpoints in $H$, and let $Z$ and $Z'$ be the irreducible components of $C_{\eta_i}$ corresponding to those endpoints, so that near the point $e$ the divisor $\on{div}(a;H)$ is given by $m_i(Z + Z')$. Now $\on{ord}_Z(\pi) = 1$, for example because the maximal ideal of the local ring at $Z$ is generated by $\pi$, and the same for $Z'$. Then near the point $e$ we have that  $\on{div}(a;H) = m_i(Z + Z') = m_i\on{div} \pi = \on{div} a$. 

It remains to treat the case where the edge $e$ is in $E(a)$, and has one vertex $v$ in $H$ and one vertex $v'$ not in $H$. 

Let $b$ denote the label of $e$, so the completed local ring at the closed point $e$ on $C$ is isomorphic to 
$$R := A[[x,y]]/(xy - b), $$
and $a = b^n$ for some $n >0$. 
Setting $X = \on{Spec} R$, let $\phi \colon X \ra C$ be the map given by choosing such an isomorphism. Without loss of generality, assume that the function $x$ vanishes on $\phi^*v$ and that $y$ vanishes on $\phi^*v'$ (here we are thinking of $v$ and $v'$ both as vertices of $\Gamma_s$ and as irreducible components of $C_s$, to avoid excessive notation). 

If we can show that $\phi^*\on{div}(a;H) = \on{div}_X x^n$ then we are done by \cref{prop:descent_of_cartier}. Note that both $\phi^*\on{div}(a;H)$ and $\on{div} x^n$ are Weil divisors on $X$ supported over $\on{div}_Sa$. Consider one of the irreducible components of $\on{div}_Sa$, say $\eta_1$. Then the specialisation map $\on{sp}\colon \Gamma_s \ra \Gamma_{\eta_1}$ does not map $v$ and $v'$ to the same vertex, since (by construction) no edges in $E(a)$ are contracted by the map $\on{sp}$.  Write $V = \on{sp} v$ and $V' = \on{sp} v'$, so $V$ and $V'$ are irreducible components of the fibre $C_{\eta_1}$. 

Writing $\bar{V}$ (resp. $\bar{V'}$) for the Zariski closure of $V$ (resp. of $V'$) in $C$, we find that $\bar{V} \supseteq v$ and $\bar{V'} \supseteq v'$  (and also $\bar{V} \not\supseteq v'$ and $\bar{V'} \not\supseteq v$), and the same holds after pulling back along $\phi$. From this we see that $x^n$ does not vanish on $\phi^*\bar{V'}$, and $y^n$ does not vanish on $\phi^*\bar{V}$. Now on $X$ we have
$$\on{div}a = \on{div}b^n = \on{div}x^n + \on{div} y^n,$$
so we see that $x^n$ vanishes on $\phi^*\bar{V}$ with the same multiplicity as $a$ (and also $y^n$ vanishes on $\phi^*\bar{V'}$ with the same multiplicity as $a$). This shows that $\phi^*\on{div}(a;H) = \on{div}_X x^n$ as required. 
%
\end{proof}

\begin{lemma}\label{lem:check_on_closed}
Let $X$ be a normal integral locally noetherian scheme, and let $D$ be a Weil divisor on $X$. Suppose that for every closed point $x$ of $X$, the pullback of $D$ to $X_x = \on{Spec}\ca{O}_{X,x}$ is Cartier. Then $D$ is Cartier on $X$. 
\end{lemma}
\begin{proof}
\shorten{The proof is straightforward and is omitted. }{Since $X$ is locally noetherian, every point of $X$ specialises to a closed point of $X$. As such, it suffices to give for every closed point $x\in X$ an open neighbourhood $U_x$ of $x$ such that the pullback of $D$ to $U_x$ is Cartier. 

Pick a closed point $x \in X$. By assumption, there is an element $f_x \in \on{Frac}\ca{O}_{X,x}$ such that on $X_x$, we have that 
\begin{equation*}
D|_{X_x} = \on{div}_{X_x}f_x. 
\end{equation*}
Now let $V$ be any affine open neighbourhood of $x$, and compare the divisors
\begin{equation*}
D|_{V} \text{ and } \on{div}_Vf_x. 
\end{equation*}
Clearly the point $x$ is not contained in the support of the difference between these divisors. Let $U_x$ be the open neighbourhood of $x$ obtained by deleting the support of that difference. Then $D|_{U_x}$ is principal, cut out by $f_x$. }
\end{proof}

Our key existence result for Cartier divisors is \cref{lem:existence_of_Cartier_divisor}. Before proving it we need some results on graph theory. 

\subsection{Graph theory 1: Cartier functions}
Let $G = (V,E,\on{ends})$ be a finite connected graph, with edges labelled by positive integers (in practice, these will be thicknesses of singularities). Analogously to \cref{def:vertex_labellings}, we say a function $f\colon V \ra \bb{Z}$ is \emph{Cartier}  if given any edge $e \in E$ (with endpoints $v_1$ and $v_2$), we have that the label of $e$ divides $f(v_1 ) - f(v_2)$. We say a function $f\colon V \ra \bb{Z}$ is \emph{elementary-Cartier} if there exists an integer $n \ge 1$ such that
\begin{enumerate}
\item
the function $f$ is constant on connected components of the graph obtained from $G$ by deleting the set of edges $\{e \in E:\on{label}(e) \text{ divides } n\}$;
\item $f$ takes values in $n \bb{Z}$. 
\end{enumerate}
Note that every elementary-Cartier function is Cartier. 
\begin{lemma}\label{lem:elementary_cartier}
Every Cartier function on $V$ can be written as an integer linear combination of elementary-Cartier functions. 
\end{lemma}
\begin{proof}Write $V = \{v_1, \cdots, v_n\}$, and let $f$ be a Cartier function. We will write $f$ as an integer linear combination of elementary-Cartier functions. Note that, by taking $n=1$, every constant function is elementary-Cartier. We may thus assume that $f(v_1) = 0$. 

Suppose $f$ vanishes on $v_1, \cdots, v_r$ for some $1 \le r < n$. We will construct a sum $g$ of elementary-Cartier functions such that $f-g$ vanishes on $v_1, \cdots, v_{r+1}$, which will yield the claim by induction. 

We are assuming that $f$ vanishes on $v_1, \cdots, v_r$, and so the same must hold for $g$. If we write $G'$ for the graph obtained from $G$ by contracting $v_1, \cdots, v_r$ to a single vertex, then clearly $f$ is a lift (along the `contraction' map $G \ra G'$) of a function on $G'$, and one also sees that any (elementary-)Cartier function on $G'$ lifts to an (elementary-)Cartier function on $G$. As such, we are reduced to the case where $r=1$. 

We have that $f(v_1) = 0$, and we need to construct an integer linear combination $g$ of elementary-Cartier functions such that $g(v_1) = 0$ and $g(v_2) = f(v_2)$. We need a bit more notation, First, let $\Omega$ denote the set of prime numbers $p$ such that $p \mid \on{label}(e)$ for some edge $e$ of $G$, i.e. 
\begin{equation*}
\Omega = \{p \text{ prime}| \exists e \in E \text{ such that } p \text{ divides }\on{label}(e)\}. 
\end{equation*}
 Given a path $\gamma$ in $G$, we write $\on{gcd}(\gamma)$ for the greatest common divisor of the labels of edges in $\gamma$. A \emph{cut} is a set $C$ of edges of $G$ such that $v_1$ and $v_2$ lie in different connected components of the graph obtained from $G$ by deleting the edges in $C$ (we write it $G\setminus C$). If $C$ is a cut, we write $\on{lcm}(C)$ for the lowest common multiple of labels of edges in $C$. 

Since $f$ is Cartier, we see that for any path $\gamma$ in $G$ from $v_1$ to $v_2$, we have that
\begin{equation*}
\on{gcd}(\gamma) \mid f(v_2), 
\end{equation*}
hence
\begin{equation*}
\on{lcm}_\gamma\on{gcd}(\gamma)  \mid f(v_2), 
\end{equation*}
where the `lcm' is over paths $\gamma$ from $v_1$ to $v_2$. 

Let $p \in \Omega$ be a prime, and let $G_p$ be a labelled graph which has the same underlying graph as $G$, but with labels given by $\on{ord}_p\on{label}_G$. 
 Let
\begin{equation*}
m_p := \on{ord}_p \on{lcm}_\gamma\on{gcd}(\gamma)   = \on{max}_\gamma \on{ord}_p \on{gcd}(\gamma) = \on{max}_\gamma \on{min}_{e \in \gamma} \ell_p(e), 
\end{equation*}
where $\gamma$ is as above. Then applying `max-flow min-cut'\footnote{
This is a trivial variation of a standard result. It is easy to prove but needs some care to extract from the literature, so we prove it here. 

Let $H$ be a finite graph with edges labelled by non-negative integers, and $h_1$ and $h_2$ distinct vertices. The \emph{flow} along a path $\gamma$ from $h_1$ to $h_2$ is the minimum of the labels along the edges in the path, and the \emph{value} of a cut $C$ separating $h_1$ and $h_2$ is the maximum of the labels on edges in the cut. Then \emph{max-flow min-cut} states that the maximum over paths $\gamma$ of the flow along $\gamma$ (`max flow') equals the minimum over cuts $C$ of the value of $C$ (`min cut'). 

\emph{Proof:} Min cut $\ge$ max flow is obvious. For the converse, let $V'$ denote the set of vertices in $H$ which can be reached from $h_1$ by traversing only edges with label at least the min cut. If $h_2 \in V'$ then we are done, and if not then the set of edges with one end in $V'$ and the other end not in $V'$ is a cut, contradicting the definition of $V'$. $\qed$} to the graph $G_p$, we see that
\begin{equation*}
m_p = \on{min}_C \on{max}_{e \in C} \ell_p(e)
\end{equation*}
where the $\on{min}$ is now over cuts. We re-write this as
\begin{equation*}
m_p = \on{min}_C \on{ord}_p \on{lcm}(C). 
\end{equation*}
Let $C_p$ be a cut achieving this minimum. We have constructed $C_p$ as a cut of $G_p$, but we can also think of it as a cut of $G$; from now on we will do so. Set $L_p = \on{lcm}(C_p)$ \emph{where we take the labels in }$G$. Define a function
\begin{equation*}
g_p\colon V  \ra \bb{Z}
\end{equation*}
to take the value $L_p$ on vertices in the same connected component of $G \setminus C_p$ as $v_2$ and 0 elsewhere. The function $g_p$ is elementary-Cartier \emph{as a function on} $G$ since the $G$-labels of edges in $C_p$ divide $L_p$. 

Note that 
\begin{equation*}
\on{gcd}\{L_p:p \in \Omega\} \mid f(v_2)
\end{equation*}
since 
\begin{equation*}
\on{ord}_p(\on{lcm}(C_p)) = m_p = \on{ord}_p(\on{lcm}_\gamma \on{gcd}(\gamma)) \le \on{ord}_pf(v_2). 
\end{equation*}
Hence there exist integers $n_p$ such that
\begin{equation*}
f(v_2) = \sum_{p \in \Omega} n_p L_p. 
\end{equation*}
Then setting 
\begin{equation*}
g = \sum_{p\in\Omega}n_pg_p, 
\end{equation*}
we see that $g$ takes the required values. 
\end{proof}
\begin{lemma}\label{lem:existence_of_Weil_divisor} Let $(S,s)$ be a regular strictly henselian local scheme, $C/S$ a generically-smooth semistable curve, and $\phi\colon T \ra S$ a non-degenerate trait through $s$. Let $a \in \ca{O}_{S,s}$. Define $\Gamma_s(a)$ as in \cref{definition_of_some_divisors}. Let $\frak{G}_1, \cdots, \frak{G}_n$ denote the connected components of $\Gamma_s(a)$. Let $0 = c_1, c_2,  \dots, c_n$ be integers, and let $m:V \ra \bb{Z}$ be the vertex labelling given by assigning to a vertex $v$ of $\Gamma_s$ the integer $c_i$ such that $v$ is contained in $\frak{G}_i$. Assume that the $c_i$ are chosen such that $m$ is $T$-Cartier. 

Then there exists a Cartier divisor $D$ on $C$, trivial over the generic point of $S$, and such that the vertex labelling of $\phi^*D$ is $m$. 
\end{lemma}
\begin{proof}
We begin by deleting all self-loops from $G$ (as they have no impact on the functions we want to construct). 

Write $S = \on{Spec}A$ and write $\on{ord}_T$ for the normalised valuation on $T$. Given an edge $e \in E(a)$, we know by definition that some power of the label of $e$ is equal to the ideal $aA$. Since $S$ is factorial, there exists $\alpha \in A$ such that for every edge $e \in E(a)$, the label of $e$ is a power of $\alpha A$. 

We are interested in functions on the vertices of $\Gamma_s$ which are constant on the $\mathfrak{G}_i$. These are canonically the same as functions on the vertices of the graph $G$ obtained from $\Gamma_s$ by contracting each $\mathfrak{G}_i$ to a point - in what follows we will sometimes not distinguish between these, but it will hopefully be clear from the context which is intended. 

Before proceeding with the proof, we will define various labelings on the edges and vertices of the graph $G$. 

\emph{Labellings on edges of G:}
\begin{itemize}
\item
$\ell_A$ the usual edge labelling taking values in the monoid of principal ideals of $A$ (but restricted to edges in $G$); 
\item $\ell_T$ taking values in $\bb{Z}$, defined by sending an edge $e$ to $\on{ord}_T\phi^*\ell_A(e)$. This coincides with (the restiction to $G$ of) the usual labelling taking values in principal ideals of $T$, followed by the $\on{ord}_T$ map to $\bb{Z}$. 
\item 
$\ell_\alpha$ defined by $\ell_\alpha(e) = \frac{\on{ord}_T\ell_A(e)}{\on{ord}_T\alpha}$ - this also takes values in $\bb{Z}$ by definition of $\alpha$. 
\end{itemize}

It is clear that for every $e$ we have that $\ell_T(e) = \ell_\alpha(e)\on{ord}_T\alpha$. 

\emph{Labellings on vertices of $G$:}
\begin{itemize}
\item
The labelling (corresponding to) $m$, taking values in $\bb{Z}$;
\item 
Now $G$ is connected, and for every edge $e$ from $u$ to $v$ we have that 
\begin{equation*}
\ell_T(e) \mid m(u) - m(v)
\end{equation*}
and also $\on{ord}_T\alpha \mid \ell_T(e)$. Using also that $m$ takes the value 0 on some vertex, we see that for every vertex $v$, the integer $m(v)$ is divisible by $\on{ord}_T\alpha$. We then define a new vertex labelling $m'$ by $m'(v) = m(v)/\on{ord}_T\alpha$, taking values in $\bb{Z}$. We see immediately that $m'$ is Cartier with respect to the labelling $\ell_\alpha$. 
\end{itemize}

Using \cref{lem:elementary_cartier} we decompose $m'$ into a sum of vertex labellings which are elementary-Cartier with respect to $\ell_\alpha$. We can then assume without loss of generality that there exists a positive integer $r$ and a set of vertices $H$ of $G$ such that $m'$ takes the value 0 outside $H$ and the value $r$ on $H$, and such that for all edges $e$ in $G$ with exactly one end in $H$ we have that $\ell_\alpha(e) \mid r$. 

Define $D:= \on{div}(\alpha^r; H)$. This is Cartier by \cref{lem:construct_cartier}. I then claim that the divisor $D$ on $C_T$  is the divisor that we seek, i.e. that the vertex labelling of $\phi^*D$ is $m = m'\on{ord}_T\alpha$. 

This is easy: first, it is clear that $\phi^*\on{div}(\alpha^r; H)$ corresponds to the 0 vertex labelling outside $H$. Let $h \in H$ be any vertex, and write $\eta_h$ for the generic point of the irreducible component of the special fibre of $C_T$ corresponding to $h$. Then by the proof of \cref{lem:construct_cartier} we see that $\phi^*\on{div}(\alpha^r; H)$ is cut out by $\alpha^r$ near $\eta_h$. Finally, we verify that 
\begin{equation*}
\on{ord}_T\alpha^r = r \on{ord}_T\alpha = m'(h) \on{ord}_T\alpha = m(h)
\end{equation*}
as required. 
\end{proof}

\subsection{Graph theory 2: achievable functions}
Let $G = (V, E, \on{ends})$ be a connected graph, with edges labelled by a set $L$ (i.e. a map $l:E \ra L$ is given). We say a function $f:V \ra \bb{Z}$ is \emph{pre-achievable} if there exists a subset $B$ of $E$ such that 
\begin{itemize}
\item
every edge in $B$ has the same label;
\item the function $f$ is locally constant on connected components of the graph obtained by deleting every edge in $B$ from $G$. 
\end{itemize}
We say a function $f:V \ra \bb{Z}$ is \emph{achievable} if it can be written as a (finite) sum of pre-achievable functions. One can show without too much difficulty that, if labels $l$ are constant on circuits in $G$, then every function $f:V \ra \bb{Z}$ is achievable. However, this is not enough for us; we also need to take into account an analogue of the condition that a vertex labelling be Cartier. In order to translate this notion into graph theory (and forget temporarily the geometric origins of the problem), we will add an additional labelling to the graph $G$; namely, we will label each edge $e\in E$ by a positive integer $n(e)$ (in addition to the labelling by symbols in $L$ given earlier). As before, we say that a function $f:V \ra \bb{Z}$ is \emph{Cartier} if for every edge $e \in E$ with endpoints $v_1$ and $v_2$, we have that $n(e)$ divides the difference $f(v_1)-f(v_2)$. We say a function $f$ is \emph{Cartier-achievable} if it can be written as a sum of Cartier pre-achievable functions. 

\begin{lemma}\label{lem:achieveable_functions}
In the setup of the previous paragraph, suppose that the labels in $L$ are constant on circuits in $G$ (we do not require that the integer labellings are constant). Then every Cartier function $f:V \ra \bb{Z}$ is Cartier-achievable, i.e. can be written $f = \sum_i f_i$ with each $f_i$ Cartier and pre-achievable. 

Suppose moreover that we pick a vertex $v_0$ such that $f(v_0) = 0$. Then it is possible to choose the $f_i$ such that $f_i(v_0) =0$ for all $i$. 
\end{lemma}
Before proving \cref{lem:achieveable_functions}, we briefly discuss the notion of a 2-vertex connected component of a graph $G$. Suppose that $G$ has no loops. Observe that any edge of a graph $G$ is contained in a 2-vertex connected subgraph (namely itself), and further that for any two 2-vertex-connected subgraphs $H_1$ and $H_2$ containing the same edge $e$, we have that $H_1 \cup H_2$ is also 2-vertex connected. As such, any edge of $G$ is contained in a unique maximal 2-connected subgraph. In particular, if $G$ contains no isolated vertices then $G$ is the union of its maximal 2-vertex-connected subgraphs. 
\begin{proof}[Proof of \cref{lem:achieveable_functions}]
We immediately reduce to the case where $G$ has no loops. We will inductively pick a sequence of subgraphs $$ v_0\in H_1 \sub H_2 \sub \cdots \sub H_n = G$$ and Cartier pre-achievable functions $f_1, \cdots, f_n$ on $G$ such that for all $i$, we have that $f|_{H_i} = \sum_{j=1}^if_j|_{H_i}$. It is clear by induction on $i$ that $f_i(v_0) =0$ for all $i$.

Pick $H_1 = G_1$ to be the 2-vertex-connected component of $G$ containing $v_0$. Then pick $H_{i+1}$ to be the union of $H_i$ with some 2-vertex-connected component $G_{i+1}$ of $G$ such that $G_{i+1} \cap H_i$ consists of exactly one vertex. Since $G$ is connected, we see that this procedure eventually exhausts the whole of $G$. 

Define $f_i$ by
$$f_i(v) = \left\{ \begin{split} f(v) - \sum_{j=1}^{i-1} f_j(v) & \text{ if } v \in H_i;\\
f(v')  - \sum_{j=1}^{i-1} f_j(v') & \text{ if } v \notin H_i
\end{split} \right.,$$
where $v'$ is chosen to be the unique vertex in $H_i$ such that $v$ and $v'$ are connected in the graph obtained from $G$ by removing every edge in $H_i$. 

We will now verify that that each function $f_i$ is Cartier and pre-achievable, and so that $f$ is Cartier-achievable. We will first treat $f_1$, then the general case by induction. 

Since $f$ is Cartier, we see that $f_1$ satisfies the Cartier divisibility condition for every edge in $H_1$. For every edge $e$ not contained in $H_1$, the function $f_1$ takes the same value at both endpoints of $e$, and hence the divisibility condition is trivially satisfied, so $f_1$ is Cartier. Since the labels on $G$ are constant on circuits, and $H_1$ is 2-vertex-connected, it follows that the labels are constant on $H_1 = G_1$. Define $B_1$ to be the set of edges in $G_1$, then it is clear that $f_1$ is constant on connected components of $G\setminus B_1$

Now we treat the general case. Choose $2 \le i \le n$, and suppose that $f_j$ is Cartier for all $j < i$. Now the function 
$$f'_i(v) = \left\{ \begin{split} f(v) & \text{ if } v \in H_i;\\
f(v') & \text{ if } v \notin H_i
\end{split} \right.,$$
(with $v'$ defined as above) is Cartier by the same argument as in the $f_1$ case, and hence $f_i$ is Cartier since a sum of Cartier functions is Cartier. 

It remains to check that $f_i$ is pre-achievable. Recalling that $G_i$ is a maximal 2-vertex-connected component and satisfies $H_i = G_i \cup H_{i-1}$, we define $B_i$ to be the set of edges in $G_i$. As before, the labels are constant on $B_i$ because the labels are constant on circuits and $G_i$ is 2-vertex-connected. Consider the connected components of $G \setminus B_i$. One of these components is $H_{i-1}$. Let $C$ be any other component. Then we see by construction of $f_i$ that it takes a constant value on $C$; the value it takes is exactly $f_i(v_C)$ where $v_C$ is the unique vertex contained in both $C$ and $G_i$. It remains to show that $f_i$ takes a constant value on $H_{i-1}$. Indeed, it takes the value 0; pick any vertex $v \in H_{i-1}$, then
\begin{equation*}
f_i(v) = f(v) - f_{i-1}(v) - \sum_{j=1}^{i-2}f_j(v) = f(v) - \left(f(v) - \sum_{j=1}^{i-2}f_j(v)\right) - \sum_{j=1}^{i-2}f_j(v) = 0. 
\end{equation*}
\end{proof}

\subsection{Proof of the main existence result}
The key existence result for Cartier divisors is:
\begin{lemma}\label{lem:existence_of_Cartier_divisor}
Let $(S,s)$ be a regular strictly henselian local scheme, and $C/S$ a semistable curve smooth over a dense open $U \sub S$. Assume that $C/S$ is aligned at $s$. Let $\phi\colon T \ra S$ be a non-degenerate trait through $s$. Let $m$ denote a $T$-Cartier vertex labelling of $\Gamma_s$ which takes the value $0$ at at least one vertex. Then there exists a Cartier divisor $D$ on $C/S$, trivial over the generic point of $S$, such that $m$ is the vertex labelling corresponding to $\phi^*D$. 
\end{lemma}

\begin{proof}
We will introduce two new labelings on $\Gamma_s$. First we will also label the edges by integers; for this, we will make a map from $\ca{O}_S/\ca{O}_S^\times$ to $\bb{N}_{0}\cup\{\infty\}$, by sending an element $a$ to $\on{ord}_T\phi^*a$, and then compose our usual labelling by elements of $\ca{O}_S/\ca{O}_S^\times$ with this map. Note that all of our edges will be labelled by positive integers. 

Next, let $L$ denote the quotient of $\ca{O}_S/\ca{O}_S^\times$ by the equivalent relation saying that $[a] = [b]$ if and only if $a^n = ub^m$ for some unit $u\in \ca{O}_S^\times$ and integers $m$, $n >0$. Composing our `usual' labelling of $\Gamma_s$ by $\ca{O}_S/\ca{O}_S^\times$ with the natural map $\ca{O}_S/\ca{O}_S^\times \ra L$ gives a labelling of the edges of $\Gamma_s$ by $L$. 

We now have three distinct (but related) labelings on the edges of $\Gamma_s$. We will denote the original labelling by elements of $\ca{O}_S/\ca{O}_S^\times$ by $\ell_{\on{orig}}$, the labelling by elements of the quotient $L$ by $\ell_L$, and the labelling by integers as $\ell_T$ (since it is the only one which depends on $T$). 

We then see that a vertex labelling of $\Gamma_s$ is $T$-Cartier if and only if it is Cartier with respect to $\ell_T$ in the graph-theory sense (cf. \cref{lem:T_cartier_labelling}). 

By the assumption that $C/S$ is aligned, the graph $\Gamma_s$ has the property that the $\ell_L$ labels are constant on circuits. As such, by \cref{lem:achieveable_functions} the $T$-Cartier vertex labelling $m$ can be written as a sum of pre-achievable $T$-Cartier functions (with respect to the two labellings $\ell_L$ and $\ell_T$) all of which take the value 0 at some vertex. 

It is then enough to show the following: let $m\colon V \ra \bb{N}_0$ be a function which is pre-achievable Cartier with respect to the two labellings $\ell_L$ and $\ell_T$ in the sense of the definition above \cref{lem:achieveable_functions}, and which takes the value 0 somewhere. Then there exists a Cartier divisor $D$ on $C$, trivial over the generic point of $S$, and such that $m$ is the vertex labelling corresponding to $\phi^*D$. 

To show this, we will apply \cref{lem:existence_of_Weil_divisor}. In order to do so, we first need to pick an element $a \in \ca{O}_S$. Well, since $m$ is pre-achieveable there exists by definition an element $l \in L$ such that $f$ is constant on the connected components of the graph obtained from $\Gamma_s$ by deleting every edge with label equal to $l$. Then there exists $a \in \ca{O}_{S}$ such that for every edge $e$ of $\Gamma_s$ with $\ell_L(e) = l$, we have that $a$ is a power of $\ell_{\on{orig}}(e)$. We fix such an $a$, and we will apply \cref{lem:existence_of_Weil_divisor} using that $a$. 

Now in order to apply \cref{lem:existence_of_Weil_divisor} we need to check two things:
\begin{enumerate}
\item
The function $m$ is $T$-Cartier; 
\item The function $m$ is constant on connected components of the graph $\Gamma_s(a)$. 
\end{enumerate}
For item 1, as remarked above, we see from the construction that a vertex labelling of $\Gamma_s$ is $T$-Cartier if and only if it is Cartier in the graph-theory sense (this depends only on $\ell_T$). For item 2, the function $m$ is by definition constant on connected components of the graph obtained from $\Gamma_s$ by deleting every edge $e$ such that $\ell_L(e) = l$. By definition of $a$ we have that $a$ is a power of $\ell_{\on{orig}}(e)$ whenever $\ell_L(e) = l$. 
%
\end{proof}

\subsection{Non-aligned curves, and non-existence of certain Cartier divisors}

\begin{lemma}\label{lem:non-existence_of_Cartier}
Let $S$ be a regular scheme, $U\sub S$ a dense open subscheme, and let $\pi\colon C\ra S$ be a semistable curve, smooth over $U$. Assume that $C/S$ is \emph{not} aligned at some $s \in S$. Then there exist a non-degenerate trait $\phi:T \ra S$ through $s$ and a Cartier divisor $D$ on $C_T$, trivial over the generic point, such that there does \emph{not} exist a Cartier divisor $E$ on $C/S$, trivial over $U$ and such that $\phi^*E$ is linearly equivalent to $D$. 
\end{lemma}
In fact, the proof will show that for \emph{every} non-degenerate trait $T$ through $s$ a Cartier divisor on $C_T$ as in the statement can be found. 
\begin{example}
Before giving the proof, it may be helpful to consider an example. Suppose $S = \on{Spec}k[[u,v]]$ with $k$ separably closed, and take $C/S$ so that the dual graph of the central fibre is a 2-gon with labels $u$ and $uv$ (so $C/S$ is not aligned and $C$ is not regular). Let $\frak{C}_0$ and $\frak{C}_1$ be the irreducible components of the central fibre.  Choose $T$ to be any non-degenerate trait through the closed point $s$ of $S$, set $d = \on{ord}_T uv$ and let $D$ be the Cartier divisor $d\frak{C}_1$ on $C_T$. Then there is an obvious candidate for $E$ as a Weil divisor: namely, the fibre of $C$ over $(u=0)$ breaks up into two irreducible components, and we could set $E = dX_1$ where $X_1$ is the component containing $\frak{C}_1$. The problem is that $E$ is not Cartier at the non-regular point (where the local equation is $xy = uv$); $X_1$ corresponds to the prime ideal $(x,u)$ assuming $x$ vanishes on $X_1$, and $(x,u)^d$ is not principal in 
\begin{equation*}
k[[u,v]][[x,y]]/(xy-uv). 
\end{equation*}
To go from this to a proof takes more work - for example, perhaps we could have made a different choice of Weil divisor $E$, and in any case we must consider more general curves. The key ingredient is \cref{CDthm:main}. 
\end{example}

\begin{proof}[Proof of \cref{lem:non-existence_of_Cartier}]
We may assume that $S$ is complete, local and has separably closed residue field. Write $s$ for the closed point, and $\frak{C}_0, \cdots, \frak{C}_n$ for the irreducible components of $C_s$. Then for each $i$, $C/S$ has a section $\sigma_i$ passing through the smooth locus of $\mathfrak{C}_i$. 

Suppose we are given two distinct components which meet at some non-smooth point $p$ --- for simplicity of notation assume they are $\frak{C}_1$ and $\frak{C}_2$ --- whose completed local ring in $C$ is isomorphic to 
$$R := \ca{O}_{S,s}[[x,y]]/(xy-a)$$
for some $a \in \ca{O}_{S,s}$. A Cartier divisor $E$ on $C$ is given locally near $p$ by an element $r \in (\on{Frac}R)^\times$. Let $f_1$, $f_2 \in \on{Frac}(\ca{O}_{S,s})^\times$ be such that $\on{div} f_i = \sigma_i^*E$. Note that the divisor $\sigma_i^*E$ is independent of the choice of section $\sigma_i$ through the smooth locus of $\frak{C}_i$ - indeed, the Cartier divisor $f^*\sigma_i^*E$ is equal to $E$ along the smooth locus of $\frak{C}_i$. If we assume that $E$ is trivial over $U$, I claim there exists $\delta \in \bb{Z}$ such that $f_1= a^\delta f_2$ (up to multiplication by units in $\ca{O}_{S,s}$). 

Our next aim is to establish the claim. To fix notation, let us assume that $x$ vanishes on $\frak{C}_1$, and $y$ vanishes on $\frak{C}_2$. Firstly, we reduce to the case where $E$ is effective by adding to $E$ the pullback of some effective Cartier divisor on $S$. By \cref{CDthm:main} (which applies by \cref{lem:non_zero_divisor}) we find $b \in \ca{O}_{S,s}$, $u \in R^\times$ and $m$, $n \in\bb{Z}_{\ge 0}$ such that $r = bux^{m}y^{n}$. 

Since the function $y$ is invertible at the generic point $\eta_1$ of $\frak{C}_1$, it follows that at $\eta_1$ the Cartier divisor $E$ is also defined by $bx^m$, and hence by $ba^m$ (recalling $xy = a$ and $y$ is invertible at $\eta_1$). Hence we have that $\sigma_1^*E = \on{div} ba^m$. A similar argument shows that $\sigma_2^*E = \on{div} ba^n$, so $\sigma_1^*E = \sigma_2^*E + \on{div}a^{m-n}$, establishing the claim.

By the assumption that $C/S$ is non-aligned at $s$, there is a circuit in the labelled graph $\Gamma_s$ with the property that the labels around the circuit do not satisfy any non-trivial multiplicative relation. Let $\frak{C}_0, \cdots, \frak{C}_N$ be the vertices around such a circuit (in order), and write $a_{i}$ for the label of the edge joining $\frak{C}_i$ to $\frak{C}_{i+1}$ (working modulo $N+1$). 



Choose any non-degenerate trait $\phi \colon T \ra S$ through $s$. Since $\phi$ is non-degenerate we see that $\phi^*a_0$ and $\phi^*a_1$ are both non-zero. Since $a_0$ and $a_1$ both vanish at the closed point $s$ (since $S$ is local) we see that $\phi^*a_0$ and $\phi^*a_1$ are both non-units. Hence we have that $\on{ord}_T\phi^*a_0>0$ and $\on{ord}_T\phi^*a_1>0$. 

Without loss of generality, let us assume that the labels $a_0$ and $a_1$ are not related - more precisely, that no non-trivial multiplicative relation holds between the principal ideals $a_0$ and $a_1$ of $\ca{O}_{S,s}$. We will now define a suitable divisor $D$ on $C_T$: let $d$ denote the product of all thicknesses of all singular points of $C_T = C\times_{S, \phi} T$, and set $D = d\frak{C}_1$.  We see by \cref{lem:T_cartier_labelling} that $D$ is Cartier. It is clear that $D$ has degree 0 over $U$, hence by \cite[9.1.2]{Bosch1990Neron-models} it has degree 0 on every fibre. Now suppose a divisor $E$ as in the hypotheses does exist. We see that $(\sigma_0)_T^*D = 0$ and we may assume that $\sigma_0^*E = 0$. With this assumption, note that if $\phi^*E$ is linearly equivalent to $D$ then it is equal to $D$ (since $\pi_*\ca{O}_C = \ca{O}_S$ by \cite[exercise 9.3.11]{Fantechi2005Fundamental-alg}, and so the only principal divisors supported on the special fibre are multiples of the fibre itself). Let $f_0, \cdots, f_N \in \ca{O}_{S,s}$ be such that $\on{div}f_i = \sigma_i^*E$. If two elements $\star$ and $\star'$ of $\ca{O}_{S,s}$ differ by multiplication by a unit we write $\star \sim \star'$, so in particular we have $f_0\sim 1$. From the above claim we know that $f_1 \sim a_0^{d_0}f_0$ for some $d_0 \in\bb{Z}$, so we have
\begin{equation}
\begin{split}
\on{ord}_T(\phi^*a_0)^{d_0} &= \on{ord}_T \phi^*f_1\\
& = \on{ord}_T \phi^*\sigma_1^*D\\
& = d, 
\end{split}
\end{equation}
and so $d_0 = d/\on{ord}_T a_0$ and hence $f_1 \sim a_0^{d/\on{ord}_T a_0}f_0$. Similar calculations show that $f_2 \sim a_1^{-d/\on{ord}_Ta_1}f_1$ and that $f_2 \sim f_3 \sim\cdots \sim f_N \sim f_0$. Combining these we find that $a_0^{d/\on{ord}_T a_0} \sim a_1^{d/\on{ord}_T a_1}$, which is a non-trivial multiplicative relation between $a_0$ and $a_1$ (since $(\on{ord}_T \phi^*a_0)(\on{ord}_T \phi^*a_1)$ divides $d$), contradicting our assumptions. 
\end{proof}


\subsection{Alignment is equivalent to flatness of the closure of the unit section in the relative Picard space}\label{sec:pic}


We recall the definition and some basic properties of the relative Picard space as in \cite{Bosch1990Neron-models}. 
\begin{definition}(\cite[8.1.2]{Bosch1990Neron-models})\label{def:Pic}
Let $S$ be a scheme, and $X/S$ an $S$-scheme. We define the \emph{relative Picard functor of $X$ over $S$} to be the fppf-sheaf associated to the functor 
\begin{equation*}
\begin{split}
\on{Sch}^{\on{op}}_S & \ra \on{Sets}\\
T & \mapsto \on{Pic}(X\times_S T). 
\end{split}
\end{equation*}
We denote it by $\on{Pic}_{X/S}$. Tensor product of line bundles gives a group structure on $\on{Pic}(X\times_S T)$ for each $T$, yielding a factorisation of this functor via abelian groups. 

Given a test scheme $T/S$, it is often useful to have a concrete way to represent elements of $\on{Pic}_{X/S}$. The simplest way to do this is to assume that $\pi\colon X\ra S$ is proper and of finite presentation, that $
\pi_*\ca{O}_X = \ca{O}_S$ universally, and that $\pi$ has a section, in which case the natural sequence
$$0 \ra \on{Pic}(T)  \ra \on{Pic}(X\times_S T) \ra \on{Pic}_{X/S}(T) \ra 0$$
is exact \cite[8.1.4]{Bosch1990Neron-models}. One can use rigificators to make this even more explicit \cite[8.1]{Bosch1990Neron-models}. Note the condition that $\pi_*\ca{O}_X = \ca{O}_S$ holds universally is true in our situation by \cite[exercise 9.3.11]{Fantechi2005Fundamental-alg}. 

There are numerous results on representability of the relative Picard functor by a scheme or an algebraic space. We are interested in the relative Picard functor of $C/S$ where $C/S$ is a semistable curve. In this case $\on{Pic}_{C/S}$ is a smooth quasi-separated algebraic space over $S$ by \cite[8.3.1]{Bosch1990Neron-models} and \cite[9.4.1]{Bosch1990Neron-models}, which apply since $C \ra S$ has reduced and connected geometric fibres, and so is cohomologically flat in dimension 0. Moreover, the fibre-wise connected component of the identity (denoted $\on{Pic}^0_{C/S}$) is a separated $S$-scheme \cite[9.4.1]{Bosch1990Neron-models}. Note that over a geometric point $s \in S$ the fibre $(\on{Pic}^0_{C/S})_s$ coincides with the sub-functor of $\on{Pic}_{C_s/s}$ consisting of line bundles having degree zero on every irreducible component of the fibre $C_s$. 

Suppose in addition to the above that $C/S$ is smooth over some schematically dense open $U \hra S$, with $U \hra S$ quasi-compact. Another important sub functor of $\on{Pic}_{C/S}$ is the space $\on{Pic}^{[0]}_{C/S}$ of `line bundles of total degree zero'. We define $\on{Pic}^{[0]}_{C/S}$ to be the scheme-theoretic closure in $\on{Pic}_{C/S}$ of $\Picz_{C_U/U}$. Its name is justified by the fact that it parametrises line bundles of degree zero on every fibre. To see this, note that the subfunctor (temporarily denoted $P^0$) of $\on{Pic}_{C/S}$ consisting of line bundles having degree zero on every fibre is an open and closed subfunctor, and clearly contains $\Picz_{C_U/U}$ and hence $\on{Pic}^{[0]}_{C/S}$. Since it is open in $\on{Pic}_{C/S}$ we have that $P^0$ is smooth over $S$, and so by \cite[th\'eor\`eme 11.10.5 (ii)]{Grothendieck1966EGA.IV.3} we have that $\Picz_{C_U/U}$ is schematically dense in $P^0$, and we are done. 

\end{definition}

\begin{theorem}\label{thm:balanced_equivalent_to_flat}
Let $S$ be a regular scheme, let $U \sub S$ be a dense open, and $C/S$ a semistable curve, smooth over $U$. Then the following are equivalent:
\begin{enumerate}
\item $C/S$ is aligned;
\item  the closure of the unit section in $\Pictdz_{C/S}$ is \'etale over $S$;
\item the closure of the unit section in $\Pictdz_{C/S}$ is flat over $S$.
\end{enumerate}
\end{theorem}

The first referee pointed out that it should be possible to replace the assumption that $S$ is regular in the above theorem by `for all smooth morphisms $T \ra S$ of relative dimension 1, the scheme $T$ is locally factorial', at the cost of lengthening some proofs. 


\begin{proof}
Since all these properties are local on the target, we immediately reduce to the case where $S$ is strictly henselian local. Then the smooth locus of $C/S$ admits sections through every connected component of every fibre, so $C/S$ is projective and every irreducible component of every fibre is geometrically irreducible, so that $\on{Pic}_{C/S}$ is a scheme by \cite[theorem 8.2.2]{Bosch1990Neron-models}. 

First we prove $(1) \implies (2)$. Suppose that $C/S$ is aligned. Write $\closure{}{e}$ for the closure of the unit section in $\Pictdz_{C/S}$, and let $p \in \closure{}{e}$ be a point. By \cref{lem:section_implies_flatness}, it is enough to construct a section from $S$ through $p$ in $\closure{}{e}$. First, let $\phi\colon T \ra \closure{}{e}$ be a non-degenerate trait through $p$; so the generic point of $T$ maps to $\closure{}{e}_U$, and the closed point maps to $p$ (such a $T$ exists by \cref{lem:non_degenerate_traits_exist}). Composing with the structure map from $\closure{}{e}$ to $S$, we obtain a non-degenerate trait $\phi_S\colon T \ra S$ and a semistable curve $C_T/T$ by pullback. Then the map $\phi$ gives a line bundle $\cl{L}$ on $C_T/T$ which is trivial over $\phi_S^*U$. To show that $\closure{}{e}$ admits a section over $S$ through $p$, we will construct a line bundle $\overline{\cl{L}}$ on $C/S$ such that $\phi_S^*\overline{\cl{L}} = \cl{L}$ and such that $\overline{\cl{L}}$ is trivial over $U$. 

Choosing a rational section of $\cl{L}$, trivial over $\phi_S^*U$, we in turn obtain a Cartier divisor $D$ on the relative curve $C_T$, with the property that the restriction of $D$ to the generic fibre is zero. We will construct a Cartier divisor $\overline{D}$ on $C/S$ which pulls back to $D$ over $T$, then define $\overline{\cl{L}} = \ca{O}_C(\overline{D})$. 

From $D$ we obtain a $T$-Cartier vertex labelling on $\Gamma_s$ as in \cref{def:vertex_labellings}. Without loss of generality, we may assume that this Cartier vertex labelling takes the value 0 somewhere. Then by \cref{lem:existence_of_Cartier_divisor} (here we use that $S$ is regular) we find a Cartier divisor $\overline{D}$ on $C$ as required. This shows $(1) \implies (2)$. 

The implication $(2) \implies (3)$ is clear. Finally we prove $(3) \implies (1)$, or rather that $\neg(1) \implies \neg (3)$. Suppose then that $C/S$ is not aligned; we will show $\closure{}{e}$ is not flat. By \cref{lem:non-existence_of_Cartier} (which again uses that $S$ is regular), we find a non-degenerate trait $T$ in $S$ through $s$ and a Cartier divisor $D$ on $C_T$, zero over the generic point of $T$, and such that there does not exist a Cartier divisor $\overline{D}$ on $C/S$ which pulls back to a divisor linearly equivalent to $D$. Now $\ca{O}_{C_T}(D)$ gives a map $T \ra \Pictdz_{C/S}$, and it maps the generic point of $T$ to a point in the unit section. Since $\closure{}{e}$ is closed in $\Pictdz_{C/S}$ and the generic point of $T$ lands in $\closure{}{e}$, we see that the image of $T$ is contained in $\closure{}{e}$. Write $t$ for the closed point of $T$ and $\phi:T \ra \closure{}{e}$ for the given map. 

Suppose now that $\closure{}{e}$ is flat over $S$; we will derive a contradiction. By \cref{lem:section_implies_flatness} we find a section $\sigma$ of $\closure{}{e} \ra S$ through $\phi(t)$. This section corresponds to a Cartier divisor $\overline{D}$ on $C/S$. This divisor is zero over $U$. Since $\closure{}{e}_U\ra U$ is an isomorphism, we find that the generic point of $T$ also maps to the image of $\sigma$. Since $T$ is reduced, we deduce that $\phi:T \ra \closure{}{e}$ factors via $\sigma$. As such, we find that the pullback of $\overline{D}$ to $T$ is linearly equivalent to the divisor $D$, a contradiction. 
\end{proof}

\begin{lemma} \label{lem:section_implies_flatness}
Let $S$ be a noetherian scheme and $U \sub S$ dense open. Let $f \colon X \ra S$ be a morphism of schemes locally of finite type and which is an isomorphism over $U$, and such that $f^{{-1}}U$ is schematically dense in $X$. Let $x \in X$ be a point. The following are equivalent:
\begin{enumerate}
\item
$f$ is \'etale at $x$;
\item
$f$ is flat at $x$;
\item there exists an open neighbourhood $V$ of $f(x)$ in S and a section $\sigma\colon V \ra X$ through $x$. 
\end{enumerate}
\end{lemma}
\begin{proof}
$(1) \implies (2)$ is obvious. 

$(2) \implies (3)$: Shrinking, we may assume $f$ is surjective and of finite type, and both $X$ and $S$ are affine, so $f$ is also separated. Then by \cite[lemma 2.0]{Lutkebohmert1993On-compactifica} we find that $f$ is an isomorphism, in particular it has a section through $x$.

$(3) \implies (1)$: Shrinking, we may again assume $X$ and $S=V$ are affine so $f$ is separated. The image of $S$ in $X$ is a closed subscheme (by separatedness of $X/S$), and contains the schematically dense subscheme $f^{-1}U$. Hence the image of $S$ in $X$ is the whole of $X$, so $f$ is an isomorphism and hence \'etale. 
\end{proof}

\section{Flatness of closure of the unit section in the Picard space is equivalent to existence of a N\'eron model}\label{sec:Neron-models}

\begin{lemma}[N\'eron models satisfy smooth descent]\label{lem:NM_smooth_descent}
Let $S$ be a scheme, $U \hra S$ a schematically dense open with $U \hra S$ quasi-compact, and $f\colon S' \ra S$ a smooth surjective morphism. Let $A/U$ be an abelian scheme. Suppose $f^*A$ has a N\'eron model $N'$ over $S'$. Then $A$ has a N\'eron model $N$ over $S$, and moreover $f^*N = N'$. 
\end{lemma}
The author would like to thank K\k{e}stutis \v{C}esnavi\v{c}ius for pointing out that some seperatedness hypotheses in an earlier version of this lemma were unnecessary. 

\begin{proof}
A N\'eron model is unique if it exists, by the universal property. The pullback of $U$ is schematically dense in $S'$ by \cite[th\'eor\`eme 11.10.5 (ii)]{Grothendieck1966EGA.IV.3}. Moreover, one the pullback of a N\'eron model along a smooth morphism is again a N\'eron model. Combining this with the effectivity of descent for algebraic spaces \cite[\href{http://stacks.math.columbia.edu/tag/0ADV}{Tag 0ADV}]{stacks-project}, we deduce that $N'$ descends to a (smooth, separated) algebraic space $N/S$. We need to check that the descended object $N$ satisfies the N\'eron mapping property. Write $S'' = S' \times _S S'$ and $q\colon S'' \ra S$. Let $T \ra S$ be a smooth morphism, and $g\colon T_U \ra A$ a $U$-morphism. By descent, it is enough to check that there is a unique $S''$ morphism $q^*T \ra q^*N$ extending $q^*g\colon q^*T_{U} \ra q^*N$. Now $q^*T \ra S''$ is smooth, in particular flat, so by \cite[th\'eor\`eme 11.10.5 (ii)]{Grothendieck1966EGA.IV.3} again we have that $q^*T_U \ra q^*T$ has scheme-theoretic closure equal to $q^*T$. Moreover, $q^*N \ra S''$ is separated, so the conclusion holds by \cite[\href{http://stacks.math.columbia.edu/tag/084N}{Tag 084N}]{stacks-project}. 
\end{proof}


\begin{theorem}\label{lem:NM_exists}
Let $S$ be a locally noetherian scheme, $U \sub S$ a schematically dense open subscheme, and $f\colon C \ra S$ a semistable curve which is smooth over $U$.  Write $e$ for the unit section in $\Pictdz_{C/S}$, and $\closure{}{e}$ for the closure of $e$ in $\Pictdz_{C/S}$ (or equivalently in $\on{Pic}_{C/S}$). Write $J$ for the jacobian of the smooth proper curve $C_U/U$. 
\begin{enumerate}
\item \label{flat_implies_NM} If $C$ is regular and $\closure{}{e}$ is flat over $S$, then a N\'eron model for $J$ exists. 

\item \label{NM_implies_flat} If a N\'eron model for $J$ exists, then  $\closure{}{e}$ is \'etale over $S$. 
\end{enumerate}
\end{theorem}

\begin{remark}
\begin{itemize}
\item
It is not really necessary for $C$ to be regular; it would for example suffice to have $C_T$ locally factorial for every smooth morphism $T \ra S$ (which is probably a weaker condition, even if $C \ra S$ is smooth). 

\item A slightly different approach to proving part 1 of this result (suggested to the author by K\k{e}stutis \v{C}esnavi\v{c}ius) would be to apply the theory of parafactorial pairs \cite[21.13]{Grothendieck1967EGA.IV.4}. 

\item By the universal property of the N\'eron model there is a canonical map from $\on{Pic}^0_{C/S}$  to the fibrewise-connected-component-of-identity of the N\'eron model, which we denote by $\cl{N}^0$. By \cite[IX,  proposition 3.1 (e)]{SGA7I} (which applies to algebraic spaces, see discussion below) this map is an open immersion, and hence by \cite[VI$_A$ 0.5]{Grothendieck1970Schemas-en-grou} it is an isomorphism. Now $\on{Pic}^0_{C/S}$ is a scheme (\cite[4.3]{Deligne1985Le-lemme-de-Gab}, \cite[9.4.1]{Bosch1990Neron-models}), and hence so is $\cl{N}^0$. 

Moreover, the identity component of $\on{Pic}_{C/S}$ (and hence of the N\'eron model) is a quasi-projective scheme, by \cite[Remarque XI 1.3(c), Th\'eor\`eme XI 1.13]{Raynaud1970Faisceaux-ample}.

\item The \'etaleness of the closure of the unit section extends a result of Raynaud \cite{Raynaud1970Specialisation-} from the case where $S$ is a trait, but under much more restrictive hypotheses on $C/S$. 

\end{itemize}
\end{remark}


\begin{proof}
Note that $\closure{}{e}$ is a subgroup space of $\Pictdz_{C/S}$. For \eqref{flat_implies_NM}, we first observe that the quotient of $\Pictdz_{C/S}$ by $\closure{}{e}$ exists as an algebraic space since $\closure{}{e}$ is flat over $S$; we denote this quotient by $\cl{N}$. It is in a canonical way a group algebraic space over $S$. We will show that $\cl{N}$ is the N\'eron model of $J$. Combining \cref{lem:NM_smooth_descent} with the fact that $C/S$ admits sections \'etale locally (since the smooth locus of $C/S$ meets every fibre), we may assume that $C/S$ has a section. 

Note that $J =  {\Pictdz_{C/S}}\times_S U= \cl{N}_U$. Since $\cl{N}$ is separated (as its unit section is a closed immersion), the uniqueness part of the N\'eron mapping property is automatic; we need to show existence. Let $T \ra S$ be a smooth morphism of spaces, and $T_U \ra \cl{N}_U$ any $S$-morphism. Let $T' \ra T$ be an \'etale morphism, with $T'$ a scheme. By base-change, we obtain a morphism $T'_U \ra {\Pictdz_{C/S}} \times_S U$. 

Since $T' \ra S$ is smooth and $C$ is regular, we see that $T' \times_S C$ is regular \cite[Tag 036D]{stacks-project}. Since $C/S$ admits a section, the Picard functor coincides with the rigidified Picard functor (rigidified along that section), so in particular there exists a line bundle $\cl{F}$ on $T'_U \times_U C_U$ such that the map $T'_U \ra {\Pictdz_{C/S}}\times_S  U$ is given by that line bundle. Let $D$ be a Cartier divisor on $T'_U \times_U C_U$ such that $\cl{O}(D) \cong \cl{F}$. 

Write $D = \sum_i n_i p_i$ where the $p_i$ are prime Weil divisors. Let $\bar{p}_i$ denote the scheme-theoretic closure of $p_i$ in $T' \times_S C$, and write $\bar{D} = \sum_i n_i \bar{p}_i$. A-priori this $\bar{D}$ is a Weil divisor, but by regularity it is in fact Cartier, so that $\cl{O}(\bar{D})|_U = \cl{O}(D) \cong \cl{F}$. Let $\bar{\cl{F}} \defeq \cl{O}(\bar{D})$, then $\bar{\cl{F}}|_U = \cl{F}$. This sheaf $\bar{\cl{F}}$ defines a morphism $T' \ra \Pictdz_{C/S}$, and by composition a morphism $T' \ra \cl{N}$, which coincides with the original map $T'_U \ra \cl{N}_U$ upon restriction to $U$. Finally, we must descend this to a map $T \ra \cl{N}$, but this follows from the uniqueness part of the N\'eron mapping property. This proves the existence part of the N\'eron mapping property. 

To complete the proof of \eqref{flat_implies_NM}, we must show that $\cl{N}$ is smooth over $S$. We know that $\on{Pic}^{[0]}_{C/S}$ is smooth and is an fppf cover\footnote{We define `fppf cover' as in \cite[\href{http://stacks.math.columbia.edu/tag/021L}{Tag 021L}]{stacks-project}, in particular it need not be quasi-compact. The result we use can also be found in \cite[Proposition 17.7.7]{Grothendieck1967EGA.IV.4}. } of $\cl{N}$, so we deduce that $N$ is smooth. 

For \eqref{NM_implies_flat}, write $\cl{N}$ for the (smooth) N\'eron model of $J$. After base-change to an \'etale cover we may assume that $\on{Pic}_{C/S}$ is a scheme (since being \'etale is even fpqc local on the target by \cite[\href{http://stacks.math.columbia.edu/tag/02YJ}{Tag 02YJ}]{stacks-project}). We have the identity ${\Pictdz_{C/S}}\times_S U = \cl{N}_U$, and hence by smoothness of $\Pictdz_{C/S}$ over $S$ and the N\'eron mapping property, we get a canonical map $\phi\colon \Pictdz_{C/S} \ra \cl{N}$. 

Suppose for a moment that the map $\phi$ is flat. Write $K$ for the kernel of $\phi$; then the canonical map $K \ra S$ is also flat, and the canonical map $K \ra \Pictdz_{C/S}$ is a closed immersion since $\cl{N}$ is separated (and so $K$ contains $\closure{}{e}$).  Moreover, the preimage of $U$ is schematically dense in $K$ by \cite[th\'eor\`eme 11.10.5 (ii)]{Grothendieck1966EGA.IV.3}, so in fact $K = \closure{}{e}$, in particular $\closure{}{e}$ is flat over $S$. Then by \cref{lem:section_implies_flatness} it is \'etale over $S$ as required. 


As such, it suffices to show that $\phi$ is flat. By the \emph{crit\`ere de platitude par fibres} (\cite[\href{http://stacks.math.columbia.edu/tag/05X0}{Theorem 05X0}]{stacks-project}) it is enough to show that for every $s \in S$, the fibre $\phi_s$ of $\phi$ at $s$ is flat. Further, it is enough to show that the restriction $\phi_s^0\colon \on{Pic}^0_{C_s/s} \ra \cl{N}_s^0$ to the connected components of identity is flat - one just covers $\on{Pic}^{[0]}_{C_s/s}$ by translates of $\on{Pic}^0_{C_s/s}$, and uses that the former is a scheme by \cite{Artin1969Algebraization-}. 

Write $\phi^0$ for the restriction of $\phi$ to the fibrewise-connected components of identity. Clearly $\phi^0$ is an isomorphism over $U$. As such\footnote{Here we apply the result to a morphism of spaces, when the reference given proves it for schemes. It turns out that the same proof works for spaces. }, \cite[IX,  proposition 3.1 (e)]{SGA7I} implies that $\phi_0$ is an open immersion, in particular flat, hence so is $\phi^0_s$, and we are done.


%
%
\end{proof}

\begin{corollary}\label{cor:NM_outside_codimension_2}
Let $S$ be an excellent scheme, regular in codimension 1, and let $U \sub S$ be dense open and regular. Let $C/S$ be a semistable curve which is smooth over $U$. Then there exists an open subscheme $U \sub V \sub S$ such that the complement of $V$ in $S$ has codimension at least 2 in $S$ and such that the Jacobian of $C_U /U$ admits a N\'eron model over $V$, whose identity component is a quasi-projective $V$-scheme. 
\end{corollary}
\begin{proof} Removing some closed codimension-2 subscheme from $S$, we may assume $S$ is regular. Alignment outside codimension 2 is clear from the definition. Existence of the N\'eron model follows immediately from \cref{thm:balanced_equivalent_to_flat} and \cref{lem:NM_exists} once we show that $C/S$ has a regular model outside some codimension 2 subscheme of $S$; the latter is \cref{prop:resolution_outside_codimension_2}. \end{proof}

\begin{proposition}\label{prop:resolution_outside_codimension_2} 
Let $S$, $U$, $C$ be as in the statement of \cref{cor:NM_outside_codimension_2}. Then there exist:
\begin{itemize}
\item
an open subscheme $U\sub V \sub S$ whose complement has codimension at least 2;
\item a modification (proper birational map) $\tilde{C} \ra C_V$ which is an isomorphism over $U$
\end{itemize}
such that $\tilde{C}$ is regular and $\tilde{C} \ra V$ is semistable. 
\end{proposition}
\begin{proof}
Throwing away codimension 2 loci, we reduce immediately to the case where $S$ is regular, and $S\setminus U$ is strict normal-crossings-divisor. By \cite[lemma 3.2]{Jong1996Smoothness-semi}, there exists a projective modification of semistable curves $\tilde{C} \ra C$ which is an isomorphism over $U$, and such that $\tilde{C}$ is regular outside some locus of codimension 3 in $\tilde{C}$. The image of that locus has codimension at least 2 in $S$; throwing it away, we are done. 
\end{proof}

\section{Non-aligned implies multiples of sections do not lift even after proper surjective base change}\label{sec:counterexample}

We have shown that the jacobian of a non-aligned semistable curve does not admit a N\'eron model, even after an alteration of the base. On the other hand, Gabber's lemma  \cite{Deligne1985Le-lemme-de-Gab} shows that any abelian scheme admits a semi-abelian prolongation after alteration\footnote{The statement in \cite{Deligne1985Le-lemme-de-Gab} requires a proper surjective morphism rather than just an alteration. However, the same proof yields the result after an alteration by replacing the reference to \cite[4.11, 4.12]{Deligne1969The-irreducibil} in lemme 1.6 by a reference to \cite[th\'eor\`eme 16.6]{Laumon2000Champs-algebriq}. Alternatively, one can deduce this version from that in \cite{Deligne1985Le-lemme-de-Gab} by a slicing argument. } of the base.

As a result, it is perhaps reasonable to hope that the situation is better if one does not attempt to construct a N\'eron model but rather considers only one section at a time. Given a scheme $S$, a dense open subscheme $U$, an abelian scheme $A/U$, and a section $\sigma \in A(U)$, one can ask whether there exists a semi-abelian prolongation $\cl{A}/S$ such that some multiple of $\sigma$ extends to a section in $\cl{A}(S)$, at least after replacing $S$ by an alteration. A positive answer to this question was incorrectly claimed by the author in a previous version of this paper on the arXiv - the error was pointed out by Jos\'e Burgos Gil. In this section, we provide an example to show that such a semi-abelian prolongation prolonging $\sigma$ will not in general exist even after proper surjective base-change of $S$. 

Construct a stable 2-pointed curve over $\bb{C}$ by glueing two copies of $\bb{P}^1_\bb{C}$ at $(0:1)$ and $(1:0)$, and marking the point $(1:1)$ on each copy. Then define $C/S$ to be the universal deformation as a 2-pointed stable curve. Choose coordinates such that $S = \on{Spec}\bb{C}[[x,y]]$, and $C$ is smooth over the open subset $U = D(xy)$. Call the sections $p$ and $q$. 

Now the graph over the closed point of $S$ is a 2-gon, with one edge labelled by $(x)$ and the other by $(y)$. The graph over the generic point of $(x=0)$ is a 1-gon with edge labelled by $(y)$, and similarly the graph over the generic point of $(y=0)$ is a 1-gon with edge labelled by $(x)$. All other fibres are smooth. In particular, $C/S$ is aligned except at the closed point.

Let $J/U$ denote the jacobian of $C_U/U$, and write $\sigma = [p-q] \in J(U)$. A \emph{pointed semi-abelian prolongation} consists of 
\begin{itemize}
\item
a proper surjective morphism $f\colon S' \ra S$;
\item a semi-abelian scheme $\cl{A}/S'$;
\item an isomorphism $\cl{A}|_{f^{-1}U} \ra f^*J$;
\item an integer $n \ge 1$;
\item a section $\tau \in \cl{A}(S')$;
\end{itemize}
such that $\tau$ extends $n\cdot f_U^*\sigma$, where $f_U$ denotes the restriction of $f$ to $f^{-1}U$. 

\begin{proposition}
No pointed semi-abelian prolongation exists. 
\end{proposition}
\begin{proof}
Suppose a pointed semi-abelian prolongation is given. Let $X$ be a prime divisor on $S'$ lying over the prime divisor $(x)$ of $S$ such that $f^*C$ is not aligned at some geometric point  $p\in X$ (exists by the easy direction of \cref{lem:pull_back_balance}). The graph of $f^*C$ over $p$ is a 2-gon, with one edge labelled by $X$. Write $Z$ for the label of the other edge. Let $\phi_i\colon T_i \ra S'$ be a sequence of non-degenerate traits in $S'$ through $P$ such that $\on{ord}_{T_i}X = 1$ for all $i$, and such that $\on{ord}_{T_i}Z$ tends to infinity as $i$ tends to infinity. For each $i$, the pullback $(\phi_i \circ f)^*J$ admits a N\'eron model; write $\sigma_i$ for the extension of the section $\sigma$ given by the N\'eron mapping property. Then a simple calculation shows that the order of $\sigma_i$ in the component group of the N\'eron model of $(\phi_i \circ f)^*J$ tends to infinity with $i$. This contradicts the existence of a pointed semi-abelian prolongation. 
\end{proof}

\bibliographystyle{alpha} 
\bibliography{../../../prebib.bib}

\end{document}